\documentclass[a4paper,11pt,oneside]{article}
\usepackage[utf8]{inputenc}   
\usepackage[english]{babel}
\usepackage{verbatim}
\usepackage{amssymb,amsfonts,amsmath,amsthm,qsymbols}
\usepackage{cancel,enumerate}
\newcounter{a}
\newenvironment{cond}{\refstepcounter{a} 
~~~~\begin{minipage}{15cm} {\bf Cond} \thea~:  }{\end{minipage}}
\def \C {{\bf Cond }}
\usepackage{graphicx}
\newcommand{\Sr}{\rho^{1/2}}
\newcommand{\tran}[1]{{}^t#1}

\def \captionn#1{\begin{center}\begin{minipage}{15cm}\caption{\small\sf #1}\end{minipage}\end{center}}
\def \mo{{\;\sf mod\;}}
\def \Proof{\noindent {\bf Proof. }}
\def \l{\left}
\def \r{\right}

\def \P{\mathbb{P}}
\def \Z{\mathbb{Z}}
\def \LL{\mathbb{L}}
\def \ME{{\sf ME}}
\newcommand{\MV}[1]{\ME(#1)} 
\newcommand{\RV}[1]{R_{#1}} 
\newcommand{\LV}[1]{L_{#1}} 
\def \Id{{\sf Id}}
\def \bT{{\bf T}}
\def \floor#1{\l \lfloor#1\r \rfloor}
\def \ben{\begin{eqnarray}}
\def \een{\end{eqnarray}}
\def \be{\begin{eqnarray*}}
\def \ee{\end{eqnarray*}}
\def \beq{\begin{equation}}
\def \eq{\end{equation}}

\def \P{\mathbb{P}}
\def \PCA{{\sf PCA}}

\def \HZ{{\sf HZ}}

\def \H{{\sf H}}

\def \Trace{{\sf Trace}}

\def \eref#1{(\ref{#1})}

\def \bma{\begin{bmatrix}}
\def \ema{\end{bmatrix}}

\def \imp{{\Rightarrow}}
\def \TT#1#2#3{T_{#1,#2\atop{#3}}}
\def \TTp#1#2#3{T'_{#1,#2\atop{#3}}}

\def \dis{\displaystyle}

\newcommand{\cyl}[1]{\Z / #1 \Z}

\newcommand{\n}{n} 

\newcommand{\IT}[3]{{#1}_{#2,#3}} 

\newcommand{\FA}{\mathcal{A}} 

\newcommand{\TA}{{\sf Tr}} 
\newcommand{\eqd}{\stackrel{d}=}

\theoremstyle{plain}
\newtheorem{theo}{Theorem}[section]
\newtheorem{lem}[theo]{Lemma}
\newtheorem{pro}[theo]{Proposition}
\newtheorem{cor}[theo]{Corollary}

\newtheorem{rem}[theo]{Remark}
\theoremstyle{definition}

\begin{document}%

\begin{center}
\Large\bf
Markovianity of the invariant distribution of probabilistic cellular automata on the line.\\
{\large \bf Jérôme Casse and Jean-Fran\c{c}ois Marckert}
\rm \\
\large{CNRS, LaBRI \\ Universit\'e Bordeaux \\
 351 cours de la Libération\\
 33405 Talence cedex, France}
 \normalsize
\end{center}

\begin{abstract} We revisit the problem of finding the conditions under which synchronous probabilistic cellular automata indexed by the line $\mathbb{Z}$, or the periodic line $\cyl{n}$, depending on 2 neighbours, admit as invariant distribution the law of a space-indexed Markov chain. Our advances concerns PCA defined on a finite alphabet, where most of existing results concern size 2 alphabet.\par
A part of the paper is also devoted to the comparison of different structures ($\mathbb{Z}$, $\cyl{n}$, and also some structures constituted with two consecutive lines of the space time diagram) with respect to the property to possess a Markovian invariant distribution.  \\
{\sf Keywords : } Probabilistic cellular automata, invariant distribution.  \\
{\sf AMS classification : Primary 37B15; 60K35 ,  Secondary 60J22; 37A25.} 
\end{abstract}

\section{Introduction}

\bf Foreword. \rm \sf Take a (random or not) colouring $X:=(X_i,i\in \mathbb{Z})\in \{0,1,\dots, \kappa\}^\mathbb{Z}$ of the line with numbers taken in $E_\kappa=\{0,1,\dots, \kappa\}$ for some $\kappa\geq 1$, and let $A$ be a probabilistic cellular automaton (PCA) depending on a neighbourhood of size $2$ with transition matrix $T=\l[T_{(a,b),c}, (a,b)\in E_\kappa^2,c\in E_\kappa\r]$. This PCA allows one to define a {\bf time-indexed} Markov chain $(X(t),t\geq 0)$ taking its values in the set of colourings of $\mathbb{Z}$, $\{0,1,\dots, \kappa\}^\mathbb{Z}$, as follows. Set, for any $t\geq 0$,  $X(t)=(X_i(t), i \in \mathbb{Z})$. At time 0, $X(0)=X$ and for any $t\geq 0$, the law ${\cal L}(X_i(t+1)~|~ X(t))={\cal L}(X_i(t+1)~|~(X_i(t),X_{i+1}(t))$, and this law is given by
 \[\P\big(X_i(t+1)=c~|~X_i(t)=a,~X_{i+1}(t)=b\big)=T_{(a,b),c}, ~~\textrm{for any }i \in \mathbb{Z}.\]
Conditionally on $X(t)$, the random variables $(X_i(t+1),i\in \mathbb{Z})$ are independent. Hence a PCA transforms $X(t)$ into $X(t+1)$ by proceeding to local, simultaneous, homogeneous, random transformations.  A natural question is to compute the set $S$ of stationary distributions of the Markov chain $(X(t),t\geq 0)$ in terms of $T$. It turns out that this question is quite difficult and very few is known on that point. \par
Some advances have been made in a particular direction: this is the case where there exists a stationary measure which is itself the law of a {\bf space-indexed} Markov chain. It is important to distinguish between both notions of Markov chain (MC) discussed here : $(X(t),t\geq 0)$ is by construction a MC taking its values in the configurations space $E_\kappa^{\mathbb{Z}}$, but when we say that the law of a MC is invariant by the PCA, we are talking about a MC with values in $E_\kappa$, indexed by the line.
 
A related question, is that of the characterization of the set of PCA for which there exists an invariant distribution which is Markov on a associated structure, called here an horizontal zigzag (defined in Section \ref{sec:HZHZ}).
In the paper we give a complete characterisation of the transition matrices $T$ having this property (Theorem \ref{theo:HZcolors}) and we provide a similar criterion for the case where the PCA is defined on $\cyl{n}$ instead (Theorem \ref{theo:HZcolors-2}). Till now, this result was known for $\kappa=1$ (the two-color case), and we give the result for $\kappa\geq 2$.\par
The property ``to have the law of a Markov process as invariant distribution'' depends on the graph on which the PCA is defined. In Section \ref{sec:RM} we compare the conditions needed to have this property with respect to the underlying graph.  \rm~\\
\centerline{------------------------------}

We start with some formal definitions.   {\it Cellular automata} (CA) are dynamical systems in which space and time are discrete. A CA is a 4-tuple $A:=(\LL, E_\kappa,N,f)$ where:
\begin{itemize}\itemsep0em 
\item $\LL$ is the lattice, the set of cells. It will be $\Z$ or $\cyl{n}$ in the paper,
\item $E_\kappa=\{0,1,\dots,\kappa\}$ for  some $\kappa \geq 1$, is the set of states of the cells, 
\item $N$ is the neighbourhood function: for $x\in\LL$, $N(x)$ is a finite sequence of elements of $\LL$, the list of neighbours of $x$; its cardinality is $|N|$. 
Here, 
\ben
N(x)= (x,x+1) \textrm{~~when }\LL=\mathbb{Z} \textrm{~~and~~}  N(x)=(x,x+1 \mo n) \textrm{ when }\LL=\cyl{n}, 
\een
\item $f$ is the \sl local rule. \rm It is a function $f:E^{|N|}_\kappa\to E_\kappa$.
\end{itemize}
The CA $A=(\LL, E_\kappa,N,f)$ defines a global function $F:E_\kappa^{\LL}\to E_\kappa^{\LL}$ on the set of configurations indexed by $\LL$. For any $S_0=(S_0(x),x\in \LL)$, $S_1=(S_1(x),x\in \LL):=F(S_0)$ is defined by
\[S_1(x)=f([ S_0(y),y \in N(x)]),~~~x \in \mathbb{L}.\] 
In words the states of all the cells are updated simultaneously. The state  $S_1(x)$ of $x$ at time 1 depends only on the states  $S_0(x)$ and $S_0(x+1)$ of its neighbours at time 0. 
Starting from configuration $\eta\in E_\kappa^{\LL}$ at time $t_0$, meaning $S_{t_0}=\eta$, the sequence of configurations 
\beq\label{eq:std}
S:=(S_t=(S(x,t),x\in \LL), t\geq t_0),
\eq
where $S_{t+1}:=F(S_{t})$ for $t\geq t_0$, forms what we call the space-time diagram of $A$.

\noindent{\it Probabilistic cellular automata} (PCA) are generalisations of CA in which the states $(S(x,t),x\in \LL,t\in \bT)$ are random variables (r.v.) defined on a common probability space $(\Omega,\FA,\P)$, each of the r.v. $S(x,t)$ taking a.s. its values in $E_\kappa$. Seen as a random process, $S$ is equipped with the $\sigma$-fields generated by the cylinders.  In PCA the local deterministic function $f$ is replaced by a transition matrix $\TA$ which gives the distribution of the state of a cells at time $t+1$ conditionally on those of its neighbours at time $t$:
\ben\label{eq:STX}
\P\l(S(x,t+1)= b ~|~ [S(y,t),y\in N(x)]=[a_1,\dots,a_{|N|}]\r) =\TA_{(a_1,\dots,a_{|N|}),b}.
\een
Conditionally on $S_t$, the states in $(S(x,t+1), x\in \mathbb{L})$ are independent.\par
The transition matrix (TM) is then an array of non negative numbers
\ben
\TA = \l(\TA_{(a_1,\dots,a_{|N|}),b}\r)_{\l((a_1,\dots,a_{|N|}),b\r) \in E_\kappa^{|N|}\times E_\kappa},
\een
satisfying $\sum_{b\in E_\kappa}\TA_{(a_1,\dots,a_{|N|}),b}=1$ for any $(a_1,\dots,a_{|N|})\in E^{|N|}_\kappa$.
 
Formally a PCA is a $4$-tuple $A:=(\LL, E_\kappa,N,\TA)$. Instead of considering $A$ as a random function on the set of configurations $E_\kappa^{\LL}$, $A$ is considered as an operator on the set of probability laws ${\cal M}(E_\kappa^{\LL})$ on the configuration space.  If $S_0$ has law $\mu_0$ then the law of $S_1$ will be denoted by $\TA(\mu_0)$: the meaning of this depends on the lattice $\LL$, but this latter will be clear from the context. The process $(S_t,t\in \bT)$ is defined as a time-indexed MC, the law of $S_{t+1}$ knowing $\{S_{t'}, t'\leq t\}$ is the same as that knowing $S_t$ only: conditionally on $S_t=\eta$, the law of $S_{t+1}$ is $\TA(\delta_{\eta})$,  where $\delta_\eta$ is the Dirac measure at $\eta$. 
A measure $\mu\in {\cal M}(E_\kappa^{\LL})$ is said to be invariant for the PCA $A$ if $\TA(\mu)=\mu$.  We will simply say that $\mu$ is invariant by $\TA$  when no confusion on the lattice $\LL$ exists.

The literature on CA, PCA, and asynchronous PCA is huge. We here concentrate on works related to PCA's only, and refer to Kari \cite{Kari} for a survey on CA (see also Ganguly \& al. \cite{Gal} and Bagnoli \cite{Bag}), to Wolfram \cite{wol94} for asynchronous PCA and to Liggett \cite{lig} for more general interacting particle systems. For various links with statistical mechanics, see Chopard \&. al. \cite{CLM}, Lebowitz \& al. \cite{Leb}. 
PCA are studied by different communities: in statistical mechanics and probability theory in relation with particle systems as Ising (Verhagen \cite{VER}), hard particles models (Dhar \cite{DH1,DH4}), Gibbs measures (\cite{DP, PLR, dkt, PYL}), percolation theory, in combinatorics (\cite{DH1,DH4,MBM,LB-M,Alb,Ma}) where they emerge in relation with directed animals, and in computer science around the problem of stability of computations in faulty CA (the set of CA form a Turing-complete model of computations),  see e.g. G\'acs \cite{Gacs}, Toom \& al. \cite{dkt}. In a very nice survey Mairesse \& Marcovici \cite{MM2} discuss these different aspects of PCA (see also the PhD thesis of Marcovici \cite{IM}).

\paragraph{Notation .} The set of PCA on the lattice $\LL$ equal to $\mathbb{Z}$ (or $\cyl{n}$) and neighbourhood function $N(x)=(x,x+1)$ (or $N(x)=(x,x+1 \mod n)$)  with set of states $E_\kappa$ will be denoted by $\PCA\l({\LL},E_\kappa\r)$. This set is parametrised by the set of TM $\{ (\TA_{(a,b),c},(a,b,c)\in E_\kappa^3)\} $. A TM $\TA$ which satisfies $\TA_{(a,b),c}>0$ for any $a,b,c\in E_\kappa$ is called a positive rate TM, and a PCA $A$ having this TM will  also be called a  positive rate PCA. The subset of $\PCA\l({\LL},E_\kappa\r)$ of  PCA with positive rate will be denoted by  $\PCA(\LL,E_\kappa)^\star$. In order to get more compact notation, on which the time evolution is more clearly represented, we will write $\TT{a}{b}{c}$ instead of $T_{(a,b),c}$. \medskip

Given a PCA  $A:=(\LL, E_\kappa,N,T)$  the first question arising is that of the existence, uniqueness and description of the invariant distribution(s) and when the invariant distribution is unique, the question of convergence to this latter. Apart the existence which is always guarantied (see  Toom \& al. \cite[Prop.2.5 p.25]{dkt}), important difficulties arise here and finally very little is known. In most cases, no description is known for the (set of) invariant distributions, and the question of ergodicity in general is not solved: the weak convergence of $T^{m}(\mu)$ when $m\to+\infty$ to a limit law independent from $\mu$ is only partially known for some TM $T$'s even when $\kappa=1$, as discussed in Toom \& al. \cite[Part 2, Chapters 3--7]{dkt}, and G\'acs \cite{Gacs} for a negative answer in general. 
Besides the existence of a unique invariant measure does not imply ergodicity (see Chassaing \& Mairesse \cite{CM}). 

For $A$ in $\PCA\l(\cyl{n},E_\kappa\r)$ the situation is different since the state space is finite. When $A\in\PCA(\cyl{n},E_\kappa)^\star$ the MC $(S_t,t\geq 0)$ is aperiodic and irreducible and then owns a unique invariant distribution which can 
be computed explicitly for small $n$, since $\mu =\TA(\mu )$ is a linear system. 

\subsection{The structures}
\label{sec:HZHZ}
We present now the \it geometric structures \rm that will play a special role in the paper. 
The $t$th (horizontal) line on the space-time diagram is 
\[H_t:=\{(x,t), x \in \mathbb{Z}\},\] 
and we write $H_t(n):=\{(x,t), x\in \cyl{n}\}$ for a line on the space-time diagram in the cyclic case. The $t$th \it horizontal zigzag \rm on the space-time diagram is 
\beq\label{eq:HZ}
\HZ_t:=\l\{ \l(\floor{x/2},t+\frac{1+(-1)^{x+1}}2\r), x \in\mathbb{Z}\r\},
\eq
as represented on Figure \ref{fig:lat}. 
\begin{figure}[ht]
\centerline{\includegraphics{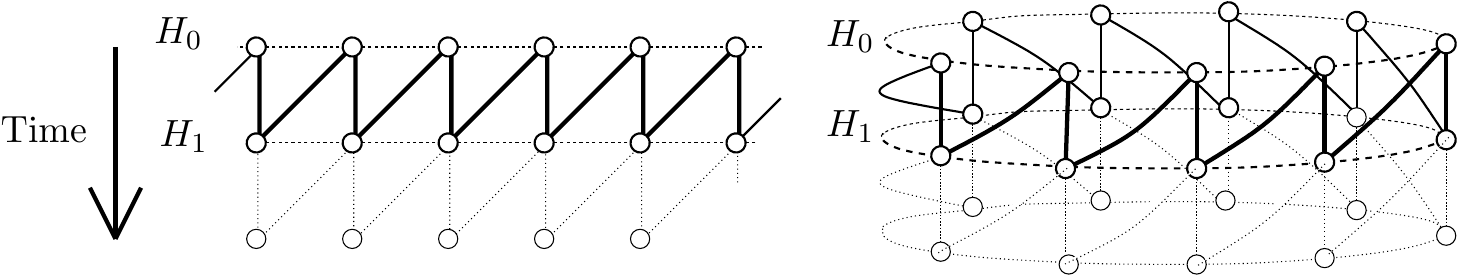}}
\captionn{\label{fig:lat} Illustration of $\HZ_t$, composed with $\H_t$ and $\H_{t+1}$, and $\HZ_t(n)$, composed by $\H_t(n)$ and $\H_{t+1}(n)$ in the case $t=0$ .}
\end{figure}
Define also $\HZ_t(n)$ by taking ($\floor{x/2} \mod n)$ instead of $\floor{x/2}$ in \eref{eq:HZ}. Since $\HZ_t$ is made by the two lines $\H_t$ and $\H_{t+1}$, a PCA $A=(\Z,E_\kappa,N,T)$ on $\Z$ can be seen as acting on the configuration distributions on $\HZ_t$. A transition from $\HZ_t$ to $\HZ_{t+1}$ amounts to a transition from $\H_{t+1}$ to $\H_{t+2}$, with the additional condition that the first line of $\HZ_{t+1}$ coincides with the second line of $\HZ_t$ (the transition probability is 0 if this is not the case) (see also the proof of Theorem \ref{thm:clzz} for more details).

\subsection{A notion of Markovianity per structure}
We first fix the matrix notation: $[\, A_{x,y}\,]_{ a\leq x,y \leq b}$ designates the square matrix with size $(b-a+1)^2$; the row index is $x$, the column one is $y$. The row vectors (resp. column vectors) will be written $[\, v_x \,]_{a\leq x \leq b}$ (resp. ${}^t[\, v_x\, ]_{a\leq x \leq b}$). \medskip

We define here what we call MC on $\H$, $\H(n)$, $\HZ$ and $\HZ(n)$. 
As usual a MC indexed by $\H_t$ is a random process $(S(x,t),x\in \Z)$ whose finite dimensional distribution are characterised by $(\rho, M)$, where $\rho$ is an initial probability distribution $\rho:=[\,\rho_a\,]_{a\in E_\kappa}$ and $M:=[\, M_{a,b}\, ]_{0\leq a,b \leq \kappa}$ a Markov kernel (MK) as follows:
\[\P(S(i,t)=a_i,  n_1\leq i \leq n_2)=\rho_{a_{n_1}}\prod_{i=n_1}^{n_2-1} M_{a_i,a_{i+1}},\]
where $\rho$ may depend on the index $n_1$. Observing what happens far away from the starting 
point, one sees that if the law of a $(\rho, M)$-MC with MK $M$ is invariant under a PCA $A$ with TM $T$ on the line, then the law of a $(\rho', M)$-MC with MK $M$ is invariant too, for $\rho'$ an invariant distribution for $M$. For short, in the sequel we will simply write $M$-MC and will specify the initial distribution when needed. 
 
A process $S_t$ indexed by $\HZ_t$ and taking its values in $A$ is said to be Markovian if there exists a probability measure $\rho:=(\rho_x, x \in E_\kappa)$ and two MK $D$ and $U$ such that, for any $n\geq 0$, any $a_i \in E_\kappa, b_i \in E_\kappa$,
\beq\label{eq:Mar-ZZ}
\P(S(i,t)=a_i,S(i,t+1)=b_i,0\leq i\leq n)=\rho_{a_0}\left( \prod_{i=0}^{n-1} D_{a_i,b_i} U_{b_i,a_{i+1}} \right) D_{a_n,b_n}
\eq
in which case $\rho$ is said to be the initial distribution. Again we are interested in shift invariant processes. We then suppose that $\rho$ is an invariant measure for the MK $DU$ in the sequel, that is $\rho=\rho DU$. We will call such a process a $(\rho,D,U)$-HZMC (horizontal zigzag MC), or for short a $(D,U)$-HZMC. HZMC correspond to some Gibbs measures on HZ (see e.g. Georgii \cite[Theo. 3.5]{geo11}).\par
 
 A process $S_t$ indexed by $\H_t(n)$ and taking its values in $E_\kappa$ is called a cyclic Markov chain (CMC) if there exists a MK $M$ such that for all $a = (a_0, \dots, a_{\n-1}) \in E_\kappa^{\n}$,
\begin{equation}  
\P(S(i,t)=a_i,i \in {\cyl{n}}) = Z_n^{-1}{\prod_{i=0}^{\n-1} M_{a_i,a_{i+1 \mo n}}}
\end{equation}where  $Z_n = \Trace(M^n)$.  The terminology \sl cyclic MC \rm is borrowed from Albenque~\cite{Alb}. Again, it corresponds to Gibbs distributions on $\H(n)$~\cite[Theo. 3.5]{geo11}. For two MK $D$ and $U$, a process $S$ indexed by $\HZ_t(n)$ and taking its values in $E_\kappa$ is said to be a $(D,U)$-cyclic MC (HZCMC) if for any $a_i \in E_\kappa, b_i \in E_\kappa$,
\beq
\P(S(i,t)=a_i,S(i,t+1)=b_i,0\leq i\leq n-1)=Z_n^{-1}\left( \prod_{i=0}^{n-1} D_{a_i,b_i} U_{b_i,a_{i+1 \mo n}} \right)
\eq
 where $Z_n  = \Trace((DU)^n)$. \par
We will also call product measure, a measure of the form $\mu(x_1,\dots,x_k)=\prod_{i=1}^k \nu_{x_i}$ for any $k\geq 1$, $(x_1,\dots,x_k)\in E_\kappa^k$, for some non negative $\nu_x$ such that $\sum_{x\in E_\kappa}\nu_x=1$.

\subsection{References and contributions of the present paper}

In the paper our main contribution concerns the case $\kappa> 1$. Our aim is to characterize the PCA, or rather the TM, which possesses the law of a MC as invariant measure. The approach is mainly algebraic.
Above we have brought to the reader attention that (different) PCA with same TM $T$ may be defined on each of the structure $\H$, $\H(n)$, $\HZ$ and $\HZ(n)$. The transitions $T$ for which they admit a Markovian invariant distribution depends on the structure. A part of the paper is devoted to these comparisons, the conclusions being summed up in Figure \ref{fig:imp}, in Section \ref{sec:RM}. 

The main contribution of the present paper concerns the full characterisation of the TM with positive rates for which there exists a Markovian invariant distribution on, on the one hand, $\HZ$ and, on the other hand, $\HZ(n)$ (some extensions are provided in Section \ref{sec:rprc}). One finds in the literature two main families of contributions in the same direction. We review them first before presenting our advances. 

The first family of results  we want to mention is the case $\kappa=1$ for which much is known.
  
\subsubsection{Case $\kappa=1$. Known results }

\label{sec:KR}
Here is to the best of our knowledge the exhaustive list of existing results concerning $\PCA$  having the law of a MC as invariant measure on  $\H$, $\H(n)$, $\HZ$ or $\HZ(n)$ for $\kappa=1$ and $N(x)=(x,x+1)$.\\
\underbar{\sf On the line $\H$:}
The first result we mention is due to  Beliaev \& al. \cite{bgm} (see also Toom \& al. \cite[section 16]{dkt}). 
A PCA $A=(\Z,E_1,N,T)\in \PCA(\mathbb{Z},E_1)^\star$  (with positive rate) admits the law of a MC on $\H$ as invariant measure if and only if (iff) any of the two following conditions hold:
 \begin{enumerate}[$(i)$]
\item $\TT000 \TT110\TT101 \TT011 = \TT111 \TT001 \TT010 \TT100,$ 
\item  $\TT011 \TT100 = \TT110 \TT001$ or $\TT101 \TT010 = \TT110 \TT001$.
\end{enumerate}
In case $(ii)$, the MC is in fact a product measure with marginal $(\rho_0,\rho_1)$, 
and 
\begin{displaymath}
\rho_0 = \begin{cases}
\frac{\TT000 \TT110 - \TT010 \TT100}{\TT000+\TT110-\TT010-\TT100} & \text{ if } \TT000+\TT110 \neq \TT010 + \TT100, \\
\frac{\TT110}{1 + \TT110-\TT010} & \text{ if } \TT000+\TT110 = \TT010 + \TT100
\end{cases}
\end{displaymath}
(the same condition is given in Marcovici \& Mairesse, Theorem~3.2 in~\cite{1207.5917}; see also this paper for more general condition, for different lattices, and for $\kappa>1$).\par
In case $(i)$, $M$ satisfies  $\TT010 \TT100 M_{1,0} M_{0,1}= \TT000 \TT110 M_{0,0} M_{1,1}$ and $M_{0,0} \TT001 = M_{1,1} \TT110$, and can be computed explicitly. \\

A PCA $A=(\Z,E_1,N,T)\in \PCA(\mathbb{Z},E_1)$  (without assuming the positive rate condition) admits the law of a MC on $\H$ as invariant measure iff  $\TT000=1$ or $\TT111=1$, or any of $(i)$ or $(ii)$ holds. This fact as well as that concerning the positive rate condition stated above can be shown using Proposition \ref{pro:PS}. This proposition provides a finite system of algebraic equations that $T$ has to satisfy, and this system can be solved by computing a Gröbner basis, which can be done explicitly using a computer algebra system like \sl sage \rm or \sl singular\rm. \par
Without the positive rate condition some pathological cases arise. Consider for example, the cases $(\TT100,\TT011)=(1,1)$ (case $(a)$) or $(\TT010,\TT101)=(1,1)$ (case $(b)$) or $(\TT001,\TT110)=(1,1)$ (case $(c)$). In these cases some periodicity may occur if one starts from some special configurations. Let $C_i$ be the constant sequence (indexed by $\mathbb{Z}$) equals to $i$, and $C_{0,1}$ the sequence $( \frac{1+(-1)^{n+1}}2, n \in \mathbb{Z}$) and $C_{1,0}$, the sequence  $( \frac{1+(-1)^{n}}2, n \in \mathbb{Z})$. It is easy to check that in case $(a)$, $(\delta_{C_{0,1}}+\delta_{C_{1,0}})/2$ is an invariant measure, in case $(b)$, $p\delta_{C_{0,1}}+(1-p)\delta_{C_{1,0}}$ is invariant for any $p\in[0,1]$, in case $(c)$, $(\delta_{C_{1}}+\delta_{C_{0}})/2$ is invariant. Each of these invariant measures are Markov ones with some ad hoc initial distribution. Case $(a)$ is given in Chassaing \& Mairesse \cite{CM} as an example of non ergodic PCA with a unique invariant measure (they add the conditions $\TT{0}01=\TT{1}11=1/2$). \medskip
 
\noindent\underbar{\sf On the periodic line $\HZ(n)$: }
(This is borrowed from Proposition~4.6 in Bousquet-Mélou \cite{MBM}). Let $A=(\cyl{\n},E_1,N,T)$ be a PCA in $\PCA(\cyl{\n},E_1)^\star$. In this case, $A$ seen as acting on $\HZ(n)$ admits a unique invariant distribution, and this distribution is that of a HZMC iff \ben\label{eq:sddpcs}
\TT000 \TT110\TT101 \TT011 = \TT111 \TT001 \TT010 \TT100.
\een
According to Lemma 4.4 in \cite{MBM}, this condition can be extended to $\PCA$ for which $T$, seen as acting on $E_\kappa^{\HZ(n)}$, is irreducible and aperiodic. The general case, that is, without assuming the positive rate condition, can be solved using Theorem \ref{theo:HZcolors-2}. The set of solutions contains all TM solutions to \eref{eq:sddpcs}, TM satisfying $\TT{0}00=1$ or $\TT111=1$, and some additional cases that depend on the size of the cylinder. \\
 
\noindent\underbar{\sf On $\HZ$: } 
Condition \eref{eq:sddpcs} is necessary and sufficient too (Theorem \ref{theo:HZcolors}) for positive rate automata.
 This result is also a simple consequence of Toom \& al. \cite[Section 16]{dkt}.\medskip

The other family of results are also related to this last zigzag case but are much more general. 
This is the object of the following section, valid for $\kappa\geq 1$. 
 
\subsubsection{Markovianity on the horizontal zigzag. Known results}
\label{sec:MHZZ}
Assume that a PCA $A=(\mathbb{Z}, E_\kappa,N,T)$ seen as acting on $\HZ$ admits as invariant distribution the law of a $(D,U)$-HZMC. Since $\HZ_t$ is made of $\H_t$ and $\H_{t+1}$, the law of $S_{t+1}$ knowing $S_t$ that can be computed using \eref{eq:Mar-ZZ}, relates also directly $(D,U)$  with $T$. From \eref{eq:Mar-ZZ} we check that in the positive rate case 
\ben\label{eq:zadz1}
\TT{a}{b}{c}= \frac{{\rho_a}}{ \rho_a}\frac{D_{a,c}U_{c,b}}{(DU)_{a,b}}= \frac{D_{a,c}U_{c,b}}{(DU)_{a,b}},
\een
where $\rho$ is the invariant distribution of the $DU$-MC  (solution to $\rho=\rho DU$).
Since the law of the $(D,U)$-HZMC is invariant by $T$, and since the MK of $S_t$ and $S_{t+1}$ are respectively $DU$ and $UD$, we must also have in the positive rate case,
\ben\label{eq:zadz2}
DU=UD.
\een

Indeed the law of $S_t$ and $S_{t+1}$ must be equal since they are both first line of some horizontal zigzags.
\begin{rem}Notice that when the positive rate condition does not hold, it may happen that the PCA can be trapped in some subsets of $E_\kappa^{\mathbb{Z}}$, and can admit a Markovian invariant distribution without satisfying \eref{eq:zadz1} for some $(D,U)$ and all $(a,b,c)$.\end{rem}
From Lemma 16.2 in Toom \& al. \cite{dkt}, we can deduce easily the following proposition:
\begin{pro}\label{pro:16.2} In the positive rate case, the law of  $(D,U)$-HZMC is invariant under $T$ iff both conditions \eref{eq:zadz1} and \eref{eq:zadz2} holds.
\end{pro}

This Proposition has a counterpart for positive rate PCA defined on more general lattices as $\mathbb{Z}^d$, with more general neighbourhood, where what are considered are the cases where a Gibbs measure defined on a pair of two (time) consecutive configurations is invariant. The analogous of \eref{eq:zadz1} in this setting connects the transition matrix with the potential of the Gibbs measure, the role played by $DU=UD$ is replaced by the quasi-reversibility of the global MC $S_0,S_1,\dots$ under an invariant distribution. Reversibility implies that only symmetric Gibbs measure appears (for a certain notion of symmetry). We send the interested reader to  Vasilyev \cite{vas}, Toom \& al. \cite[section 18]{dkt}, Dai Pra \& al. \cite{PLR}, PhD thesis of Louis \cite{PYL} (see also Marcovici \cite[section 1.4]{IM}) for additional details. \par
These notions of reversibility and quasi-reversibility when the ambient space is $\mathbb{Z}$ are crucial here, since PCA having this property (for some initial distributions) correspond to those for which our main Theorems \ref{thm:clzz} and \ref{theo:HZcolors} apply. We discuss this longer in Section \ref{sec:Gibbs}.

\subsubsection{Content}

Some elementary facts about MC often used in the paper are recalled in Section \ref{sec:App}.
Section \ref{sec:MHZ} contains Theorem \ref{theo:HZcolors} which gives the full characterisation of PCA with positive rate (and beyond) having a Markov distribution as invariant measure on $\HZ$. It is one of the main contributions of the paper. This goes further than Proposition \ref{pro:16.2}  (or Theorem \ref{thm:clzz}) since the condition given in Theorem \ref{theo:HZcolors} is given in terms of the transition matrix only. This condition is reminiscent of the conditions obtained in mathematical physics to obtain an integrable system, conditions that are in general algebraic relations on the set of parameters.
Theorem \ref{theo:extension} extends the results of Theorem \ref{theo:HZcolors} to a class of PCA having some non positive rate TM.   \par
Section \ref{sec:MHZn} contains Theorem \ref{theo:HZcolors-2} which gives the full characterisation of PCA with positive rate having a Markov distribution as invariant measure on $\HZ(n)$.\par

The rest of Section \ref{sec:ACM} is devoted to the conditions on $T$ under which Markov distribution are invariant measure on $\H$ and $\H(n)$. Unfortunately the condition we found are stated under some (finite) system of equations relating the TM $T$ of a PCA and the MK of the Markov distributions. Nevertheless this systematic approach sheds some lights on the structure of the difficulties: they are difficult problems of algebra! Indeed the case that can be treated completely, for example the case where the invariant distribution is a product measure and the TM $T$ symmetric (that is for any $a,b,c$, $\TT{a}{b}{c}=\TT{b}ac$) need some algebra not available in the general case. The present work leads to the idea that full characterisations of the TM $T$ having Markov distribution as invariant measure on $\H$ and $\H(n)$ involve some combinatorics (of the set $\{(a,b,c)~: \TT{a}{b}{c}=0 \}$) together with some linear algebra considerations as those appearing in Proposition \ref{pro:PS} and in its proof, and in Section \ref{sec:MID}.\par
In Section \ref{sec:RM} we discuss the different conclusions we can draw from the Markovianity of an invariant distribution of a PCA with TM $T$ on one of the structure $\H$, $\H(n)$, $\HZ$ and $\HZ(n)$, on the other structures (which is summed up in Figure \ref{fig:imp}). Apart the fact that this property on $\HZ$ implies that on $\H$ (and $\HZ(n)$ implies that on $\H(n)$) all the other implications are false.\par
Last, Section \ref{sec:PTHZ} is devoted to the proof of Theorems \ref{theo:HZcolors}, \ref{theo:extension} and \ref{theo:HZcolors-2}.

\section{Algebraic criteria for Markovianity}
\label{sec:ACM}

\subsection{Markov chains: classical facts and notation}
\label{sec:App}

We now recall two classical results of probability theory for sake of completeness.
\begin{pro}\label{pro:PF}[Perron-Frobenius]
Let $A = [\, a_{i,j}\,]_{1\leq i,j \leq n}$  be an $n \times n$ matrix with positive entries and $\Lambda=\{\lambda_i, 1\leq i \leq n\}$ be the multiset of its eigenvalues. Set $m=\max |\lambda_i|>0$ the maximum of the modulus of the eigenvalues of $A$. The positive real number $m$ is a simple eigenvalue for $A$ called the Perron eigenvalue of $A$; all other eigenvalues $\lambda \in \Lambda \setminus \{m\}$ satisfy $|\lambda|<m$. 
The eigenspace associated to $m$ has dimension 1, and the associated left (resp. right) eigenvectors $L= [\, \ell_i\,]_{1\leq i \leq n}$ (resp. $R={}^t [\, r_i\,]_{1\leq i \leq n}$) can be normalised such that its entries are positive. 
We have $\lim_{k \rightarrow \infty} A^k/m^k = RL$ for $(L,R)$ moreover normalised so that $LR$ = 1. We will call Perron-LE (resp Perron-RE) these vectors $L$ and $R$. We will call them stochastic Perron-LE (or RE) when they are normalised to be probability distributions.  We will denote by $\ME(A)$ the maximum eigenvalue of the matrix $A$, and call it the Perron eigenvalue.
\end{pro} 
One can extend this theorem to matrices $A$ for which there exists $k\geq 1$ 
such that all coefficients of $A^k$ are positive. These matrices are called \it primitive \rm in the literature. 
\begin{pro}\label{pro:invariante-distribution}  
Let $P$ be a MK on $E_{\kappa}$, for some $\kappa\geq 1$ with a unique invariant measure $\pi$; this invariant measure can be expressed in terms of the coefficients of $P$ as follows: 
\[\pi_y = \frac{\det\l((\Id^{[\kappa]}-P)^{\{y\})}\r)} {\sum_x \det\l((\Id^{[\kappa]}-P)^{\{x\})}\r)},\]
where  $P^{\{y\}}$  stands for $P$ where have been removed the $y$th column and row, and where $\Id^{[\kappa]}$ is the identity matrix of size $\kappa+1$.
\end{pro}

\subsection{Markovianity of an invariant distribution on $\HZ$: complete solution}
\label{sec:MHZ}
Here is a slight generalisation of Proposition \ref{pro:16.2}. It gives a condition for the law of a $(D,U)$-HZMC to be invariant by $T$ in terms of the 3-tuple $(D,U,T)$. 
\begin{theo} \label{thm:clzz}
Let $A:=(\Z, E_\kappa,N,T)$ be a PCA seen as acting on  ${\cal M}(E_\kappa^{\HZ})$ and $(D,U)$ a pair of MK such that for any $0\leq a,b \leq \kappa$, $(DU)_{a,b}>0$. 
The law of the $(D,U)$-HZMC is invariant by $A$ iff the two following conditions are satisfied:\\~\\
\begin{cond}\label{cond:DUT}
$\displaystyle \left\{\begin{array}{l} - \ \textrm{if } \TT{a}bc>0 \textrm{ then }\TT{a}{b}{c} =\dis\frac{D_{a,c}U_{c,b}}{(DU)_{a,b}},\\
- \ \textrm{if } \TT{a}bc=0 \textrm{ then } \TT{a}bc=D_{a,c}U_{c,b}=0.
\end{array}\right.$
\end{cond}~\\~\\
\begin{cond}\label{cond:commut}
$DU=UD$.
\end{cond}
\end{theo}
Lemma 16.2 in Toom \& al. \cite{dkt} asserts that if two MK $D$ and $U$ (with positive coefficients) satisfy $DU=UD$, then the $DU$-HMC is stable by the TM $T$ defined by
$\TT{a}{b}{c}=D_{a,c}U_{c,b}/(DU)_{a,b}.$
These authors do not consider HZMC but only MC. Vasilyev \cite{vas} considers a similar question, expressed in terms of Gibbs measure (see discussion in Section \ref{sec:Gibbs}).
\begin{rem} \label{rem:cas_taux_positif}
If $T$ is a positive rate TM then if the law of a HZMC with MK $M=DU$ is invariant by $T$ then $M_{a,b}>0$ for any $0\leq a,b\leq \kappa$ since any finite configuration has a positive probability to occur at time  1 whatever is the configuration at time 0. If a product measure $\rho^\mathbb{Z}$ is invariant then  $\rho_a>0$ for  any $0\leq a\leq \kappa$.
\end{rem}
\begin{rem}\label{rem:dazq} 
$\bullet$ Under \C \ref{cond:DUT}, if for some $a,b,c$ we have
$\TT{a}bc=0$ then either all the $\TT{a}{b'}{c}=0$ for $b'\in E_\kappa$ or all the $\TT{a'}{b}{c}=0$ for $a'\in E_\kappa$. \\
$\bullet$ Notice that the we do not assume the positive rate condition but something weaker $(DU)_{a,b}>0$; under this condition, the $DU$-MC admits a unique invariant distribution.\par
 Without the condition $(DU)_{a,b}>0$, for any $a,b$, some problems arise. Assume the law of a $(D,U)$-HZMC is invariant under $T$ but $(DU)_{a,b}=0$. Under the invariant distribution, the event $\{S(i,t)=a, S(i+1,t)=b\}$ does not occur a.s., and then the transitions $\left(\TT{a}bc, c \in E_\kappa\right)$ do not matter. For this reason, they do not need to satisfy \C \ref{cond:DUT}. In other words the condition $(DU)_{a,b}>0$ implies that each transition $\TT{a}{b}{x}$ will play a role (for some $x$). Without this condition ``pathological cases'' for the behaviour of PCA are possible as discussed in Section \ref{sec:KR}.  For example if $\TT{a}{a}a=1$ the constant process $a$ is invariant. Hence sufficient conditions for Markovianity can be expressed on only one single value $\TT{a}bc$ and only few values of $D$ and $U$ (if $\TT{1}{1}{1}=1$, $D_{1,1}=U_{1,1}=1$, the additional conditions $DU=UD$ and  $D_{a,c}U_{c,b}/(DU)_{a,b}=\TT{a}bc$  for $(a,b,c)\neq (1,1,1)$ are not needed). Further,  designing necessary and sufficient conditions for general PCA is rather intricate since a PCA on $E_\kappa$ can be trapped on a subset of $E_\kappa^\mathbb{Z}$, subset which may depend on the initial configuration. Stating a necessary and sufficient condition for a PCA to possess the law of a MC as invariant distribution, is equivalent to state a  necessary and sufficient condition for the existence of ``one trap'' with this property. When $\kappa$ grows, the combinatorics of the set of traps becomes more and more involved. 
\end{rem}

\begin{proof}[Proof of Theorem \ref{thm:clzz}] 
Assume first that $S_0$ is a $(D,U)$-HZMC, whose distribution is invariant by $A$. Using the argument developed in Section \ref{sec:MHZZ} we check that    $\TT{a}{b}{c} =\dis\frac{D_{a,c}U_{c,b}}{(DU)_{a,b}}$ 
(again when $(DU)_{a,b}>0$, the invariant law of the $DU$-MC has full support) and that $DU=UD$.\par  
Assume now that \C \ref{cond:DUT}  and \C \ref{cond:commut} hold for $D$ and $U$ (with $(DU)_{a,b}>0$, for any $a,b$).  Let us show that the law of the $(D,U)$-HZMC is invariant by $A$. 
For this start from the $(D,U)$-HZMC on $\HZ_0$, meaning that for any $a_i,b_i \in E_\kappa$,
\[\P(S(i,0)=a_i,i=0,\dots,n+1, S(i,1)=b_i,i=0,\dots,n)=\rho_{a_0}\prod_{i=0}^n D_{a_i,b_i}U_{b_i,a_{i+1}},\]
and let us compute the induced distribution on $\HZ_1$. Assume that the configuration on $\HZ_1$ is obtained by a transition of the automata from $\HZ_0$
\be
\P\l(
     \begin{array}{l}
      {S(i,1)=b_i,i=0,\dots,n},\\
      {S(i,2)=c_i,i=0,\dots,n-1}
     \end{array}
   \r)
&=&\sum_{(a_i, 0\leq i \leq n+1)}\rho_{a_0} \l(\prod_{i=0}^n D_{a_i,b_i}U_{b_i,a_{i+1}}\r)\l(\prod_{i=0}^{n-1}\TT{b_i}{b_{i+1}}{c_i}\r)\\
&=& \l(\sum_{a_0} \rho_{a_0}D_{a_0,b_0}\r)\l(\prod_{i=0}^{n-1} \sum_x (U_{b_i,x}D_{x,b_{i+1}})\r)\l(\sum_x U_{b_{n},x}\r) \\
&& \times \prod_{i=0}^{n-1}  \l(\frac{D_{b_i,c_i}U_{c_i,b_{i+1}}}{(DU)_{b_i,b_{i+1}}}\r)
\ee
The first parenthesis equals $\rho_{b_0}$, the second $\displaystyle \prod_{i=0}^{n-1}(UD)_{b_i,b_{i+1}}$, the third 1, and the denominator of the fourth simplify when multiplied by the second since $DU=UD$. This gives the desired result.
\end{proof}

We now define some quantities needed to state Theorem \ref{theo:HZcolors}.\par
Let $\nu:=\nu[T]$ be the stochastic Perron-LE of the stochastic matrix 
\[Y:=Y[T]=\bma\dis\TT{i}{i}{j}\ema_{0\leq i,j\leq \kappa}\]
and $\gamma:=\gamma[T]$ be the  stochastic Perron-LE of the matrix
\[X:=X[T]=\bma \dis\frac{\TT{a}{a}{0}\,\nu_a}{\TT{a}{d}{0}}\ema_{0\leq d,a\leq \kappa}\]
associated with $\lambda:=\lambda[T]>0$ the Perron-eigenvalue of $X$ (this matrix is defined in the positive rate case). Then $(\gamma_i, 0\leq i \leq \kappa)$ is solution to:
\ben\label{eq:qd}
\sum_d   \frac{\gamma_d}{\TT{a}{d}{0}}=\lambda\frac{\gamma_a}{\TT{a}{a}{0}\nu_a}.
\een
By Proposition \ref{pro:invariante-distribution},  $\nu$ and $\gamma$ can be computed in terms of $T$ (but difficulties can of course arise for effective computation starting from that of $\lambda$).
Define further for any $\eta=(\eta_a, 0\leq a \leq \kappa)\in {\cal M}^\star(E_\kappa)$ (law on $E_\kappa$ with full support), the MK $D^{\eta}$ and $U^{\eta}$:
 \ben\label{eq:iqb}
D_{a,c}^{\eta}=\frac{\sum_\ell \frac{  \eta_{\ell}}{\TT{a}{\ell}{0}}\TT{a}{\ell}{c} }{\sum_{b'}\frac{\eta_{b'}}{\TT{a}{b'}{0}}},~~U_{c,b}^\eta       = \frac{ \frac{\eta_{b}}{\TT{0}{b}{0}}\TT{0}{b}{c}}{ \sum_{b'} \frac{ \eta_{b'}}{\TT{0}{b'}{0}}\TT{0}{b'}{c}  },\textrm{ for } 0\leq a,b,c \leq \kappa.
\een
The indices are chosen to make easier some computations in the paper. In absence of specification the sums are taken on $E_\kappa$.
\begin{theo}\label{theo:HZcolors}  
 Let $A:=(\Z, E_\kappa,N,T)\in \PCA(\mathbb{Z},E_\kappa)^\star$ be a positive rate PCA seen as acting on ${\cal M}(E_\kappa^{\HZ})$. The PCA $A$ admits the law of a HZMC as invariant distribution on $\HZ$ iff $T$ satisfies the two following conditions:~\\
\begin{cond}\label{cond:gibbs-1}
for any $ 0\leq a,b,c\leq \kappa, \TT{a}bc=\dis\frac{\TT000 \TT{a}b0\TT{a}{0}c\TT{0}{b}c }{\TT{a}00\TT{0}b0\TT{0}0c}$,
\end{cond}~\\~\\  
 \begin{cond}\label{cond:UDDU}
the equality $D^ \gamma U^ \gamma=U^ \gamma D^ \gamma$ holds (for $\gamma$ defined in \eref{eq:qd}, and $D^\gamma$ and $ U^\gamma$ defined in  \eref{eq:iqb}).
\end{cond}~\\~\\
In this case the  $(D^\gamma,U^\gamma)$-HZMC is invariant under $A$ and the common invariant distribution for the MC with MK  $D^\gamma, U^\gamma , D^\gamma U^\gamma$ or $U^\gamma D^\gamma$  is $\rho=[\, \gamma_i \mu_i\,]_{0\leq i \leq \kappa}$  where $\mu={}^t [\,\mu_i \,]_{0\leq i \leq \kappa}$ is the Perron-RE of $X$ normalised so that $\rho$ is a probability distribution. \\
When $\kappa=1$ (the two-colour case), when  \C \ref{cond:gibbs-1} holds, then so does \C \ref{cond:UDDU}, and then the only condition is \C \ref{cond:gibbs-1} (which is equivalent to \eref{eq:sddpcs}).
\end{theo} 
Even if \C \ref{cond:UDDU} seems much similar to \C \ref{cond:commut}, it is not! In Theorem \ref{thm:clzz} the question concerns the existence or not of a pair $(D,U)$ satisfying some conditions. In Theorem \ref{theo:HZcolors} the only possible pair $(D,U)$ is identified, it is $(D^\gamma,U^\gamma)$, and the question reduces to know  if the equality $D^\gamma U^\gamma=U^\gamma D^\gamma$ holds or not. \par  

The proof of this theorem is postponed to Section \ref{sec:PTHZ}, as well as the fact that for $\kappa=1$, \C~\ref{cond:UDDU} disappears whilst this fact is far to be clear at the first glance. An important ingredient in the proof is Lemma \ref{lem:C1C2} which says that for a given $T$, \C \ref{cond:gibbs-1} is equivalent to the existence of a pair of MK $(D,U)$ such that \C \ref{cond:DUT} holds.
 \par
By \eref{eq:qd}, \C \ref{cond:UDDU}  can be rewritten 
 \ben
 \label{eq:ead-2}
 \sum_c   \frac{\frac{\TT{c}{c}{0}\nu_c}{\TT{0}{c}{0}}\TT{0}{c}{a}}
               {\l(\sum_{b'} \frac{ \gamma_{b'}}{\TT{0}{b'}{0}}\TT{0}{b'}{a}\r)  }
          \l(\sum_d \frac{  \gamma_{d}}{\TT{c}{d}{0}}\TT{c}{d}{b} \r) 
          =\frac{\TT{a}{a}{0}\nu_a}{\gamma_a}  
           \frac{  \gamma_{b}}{\TT{a}{b}{0}},~~~ \textrm{ for any }0\leq a,b \leq \kappa.
 \een

\begin{rem} Condition \C \ref{cond:gibbs-1} is bit asymmetric. In Lemma \ref{lem:C1C2} we will show that this condition is equivalent in the positive rate case to the following symmetric condition:~\\~\\
\begin{cond}\label{cond:gibbs-g}
\textrm{  for any } $0\leq  a,a',b,b',c,c', \leq \kappa$, \\
\centerline{$\TT{a'}{b'}{c'} \TT{a}{b'}{c} \TT{a}{b}{c'} \TT{a'}{b}{c} = \TT{a}{b}{c} \TT{a}{b'}{c'} \TT{a'}{b'}{c} \TT{a'}{b}{c'}$.}
\end{cond} 
\end{rem}

The following Section, whose title is explicit, discuss some links between our results and the literature.
 
\subsubsection{Reversibility, quasi-reversibility and Theorems \ref{thm:clzz} and \ref{theo:HZcolors}}
\label{sec:Gibbs}
Here is a simple corollary of Theorem \ref{thm:clzz}.
\begin{cor}\label{cor:DUTD} Let $A$ be a positive rate  PCA with TM $T$ and $(D,U)$ a pair of MK. If $(D,U,T)$ satisfies \C \ref{cond:DUT} and \C \ref{cond:commut}, then so does $(U,D,T')$, for 
 $\TT{a}{b}{c}' =\dis\frac{U_{a,c}D_{c,b}}{(UD)_{a,b}}$. 
As a consequence both $T$ and $T'$ let invariant the law $\nu$ of the MC with MK $M=DU=UD$ (under its stationary distribution).
\end{cor}
As we have explained at the beginning of the paper, a PCA $A$ with TM $T$ allows one to define a MC $(S_t[T],t\geq 0)$ on the set of global configurations (we write now $S_t[T]$ instead of $S_t$). Under the hypothesis of Corollary \ref{cor:DUTD}, we see that for any finite $n$, the time reversal $(S_{n-t}[T],0\leq t\leq n)$ of the MC $(S_t[T],0\leq t\leq n)$ starting from $S_0[T]\sim \nu$ (as defined in Corollary \ref{cor:DUTD}), is a MC whose kernel is that of the PCA with TM $T'$, whose initial law is also $\nu$. 

Let $\mu$ be a distribution on $E_\kappa^{\mathbb{Z}}$.  The MC $(S_t[T],t\geq 0)$ with initial distribution $\mu$ is reversible if the equality $(S_0[T],S_1[T])\eqd (S_1[T],S_0[T])$ holds. It is said to be quasi-reversible (see e.g. Vasilyev \cite{vas}) if the two following properties hold:\\
(a) $S_1[T]\sim \mu$,\\
(b) there exists a certain PCA $A'$ with TM $T'$ for which the distribution ${\cal L}(S_0[T]~|~S_1[T])$ (time reversal) coincides with ${\cal L}(S_1[T']~|~S_0[T'])$ (usual time). \par

Clearly, reversibility implies quasi-reversibility. Moreover, the present notion of reversibility is the same as the usual one for MC. Quasi-reversibility implies that the time reversal of the MC $(S_t[T],0\leq t\leq n)$ (for some finite $n$) where $S_0[T]\sim \mu$, is a MC whose kernel is that of a PCA $A'$ with some TM $T'$. It is then more restrictive that the only fact that the time-reversal of $S_t[T]$ is a MC. 

Theorem 3.1 in  Vasilyev  \cite{vas}  holds for PCA built on more general graphs and neighborhoods. He states that quasi-reversibility for the MC  $(S_t[T],t\geq 0)$ with initial distribution $\mu$ is equivalent to the fact that the distribution of $(S_0[T],S_1[T])$ is a Gibbs measure on a graph $\bar{\bar{\Gamma}}$ built using two copies of the graph $\Gamma$ on which is defined the PCA. For PCA build on $\mathbb{Z}$, the corresponding graph $\bar{\bar{\Gamma}}$ is simply HZ. 
Vasilyev \cite[Cor. 3.2 and Cor.3.7]{vas} characterizes the positive rate TM $T$ that induces a reversible MC (under some invariant distribution), and those quasi-reversible for general $\bar{\bar{\Gamma}}$. For the cases studied in the present paper (line case, neighborhood of size 2), a change of variables allows one to check that these cases correspond exactly to the set of $T$ which satisfy the two conditions of Theorem \ref{thm:clzz}. 

Hence, by Corollary \ref{cor:DUTD}, we can deduce that when $(D,U,T)$ satisfies Theorem \ref{thm:clzz}, then the MC $(S_t[T],t\geq 0)$ with initial distribution $\nu$ is quasi-reversible. By Corollary 3.7 in \cite{vas}, one sees that every quasi-reversible MC $(S_t[T],t\geq 0)$ with initial distribution $\mu$ satisfies the hypothesis of Theorem \ref{thm:clzz} for some $(D,U,T)$. Hence, the invariant distribution $\nu$ is the law of a MC on the line (deduction already made in Vasilyev  \cite{vas}). In fact, Vasilyev  \cite{vas} expresses his results in terms of Gibbs measures on HZ instead of HZMC, but with some changes of variables, one can pass from the first one to the other (see also  Georgii \cite[Theo. 3.5]{geo11} for more details on the correspondence between Gibbs measure on a line and MC).
 
The reversible cases treated in Corollary 3.2 in \cite{vas} correspond to the cases where $(D,U,T)$ satisfies Theorem \ref{thm:clzz}, and $D=U$. From what is said above, Theorem \ref{theo:HZcolors} applies then to the quasi-reversible case exactly.

\subsubsection{Relaxation of the positive rate condition}
\label{sec:rprc}

We will not consider all PCA that admit some Markov invariant distribution on $\HZ$ here, but only those for which the invariant distribution is the law of a HZMC of MK $(D,U)$ satisfying, for any $a,b \in E_\kappa$, $(DU)_{a,b}>0$ (this is one of the hypothesis of Theorem \ref{thm:clzz}). 
We will assume that for $i=0$, for any $a,b,c$,  $\TT{a}bi>0$ and $\TT{i}ic > 0 $ (one can always relabel the elements of $E_\kappa$ if the condition holds for a $i\neq 0$ instead).\par~\\
\begin{cond}\label{cond:tauxg}
$\textrm{for any }a,b,c,\in E_\kappa,\ \TT{a}b0>0,\ \TT{0}0c > 0.$
\end{cond}\par~\\
Under \C \ref{cond:tauxg}, \C \ref{cond:gibbs-1} is well defined.
For any pair $(a,b)$, since $\sum_c \TT{a}bc=1$, there is a $c$ such that $\TT{a}{b}c>0$; hence \C \ref{cond:gibbs-1} implies that $\TT{a}{b}0>0$ for any $a,b$. It follows that $U_{c,b}^\gamma$ and  $D_{a,c}^\gamma$ as defined in \eref{eq:iqb} are still well defined and $ (D^\gamma U^\gamma)_{a,b}>0$ for all $a,b$.
 We have the following result 
\begin{theo}\label{theo:extension}
Theorem \ref{theo:HZcolors} holds if instead of considering $A:=(\Z, E_\kappa,N,T)$ in $\PCA(\mathbb{Z},E_\kappa)^\star$, $T$ satisfies  \C \ref{cond:tauxg} instead, with the slight modification that $U_{c,b}^\gamma=0$ when 
$\l\{\TT{a}bc, a \in E_\kappa\r\}=\{0\}$.
\end{theo}
The proof is postponed to the end of Section \ref{sec:PTHZ}. It is similar to that of Theorem \ref{theo:HZcolors}.  
 
\subsection{Markovianity of an invariant distribution on $\HZ(n)$: complete solution}  
\label{sec:MHZn}
 
In the cyclic zigzag, we have
\begin{theo} \label{thm:cczz}  
Let $A:=(\cyl{n}, E_\kappa,N,T)$ be a PCA seen as acting on ${\cal M}(E_\kappa^{\HZ(n)})$ for some $n\geq 1$ and $(D,U)$ a pair of MK such that for any $0\leq a,b \leq \kappa$, $(DU)_{a,b}>0$. The law of a $(D,U)$-HZCMC on $\HZ(n)$ is invariant by $A$   iff 
\C \ref{cond:DUT} holds and \\~\\
\begin{cond}\label{cond:commut-cyl} 
$\displaystyle
  {\sf Diagonal}((DU)^k)={\sf Diagonal}((UD)^k) \textrm{~for all } 1\leq k\leq \min(\kappa+1,n).
$
\end{cond}
\end{theo}~\\

Notice that \C \ref{cond:commut-cyl} is equivalent to the fact that for all $j \leq |E_\kappa|$, for all $a_0,\dots,a_{j-1} \in E_\kappa$,  
\[\prod_{i=0}^{j-1}(DU)_{a_i,a_{i+1 \mo j}} = \prod_{i=0}^{j-1} (UD)_{a_i,a_{i+1\mo j}}.\]
It does not imply $DU=UD$ (but the converse holds). 
\begin{proof} Suppose that the law of the  $(D,U)$-HZCMC on $\HZ(n)$ is invariant by $T$. The reason why \C \ref{cond:DUT} holds is almost the same as in Section \ref{sec:MHZZ}:
\[\P(S(0,1)=c | S(0,0)=a,S(1,0)=b)=\frac{D_{a,c}U_{c,b}((DU)^{n-1})_{b,a}}{(DU)_{a,b}((DU)^{n-1})_{b,a}}= \frac{D_{a,c}U_{c,b}}{(DU)_{a,b}}.\] 
If $S$ is a $(D,U)$-HZCMC on $\HZ(n)$ then $S|_{\H_0(n)}$ and $S|_{\H_1(n)}$ are respectively $DU$ and $UD$ CMC on $\cyl{n}$. Moreover the laws of $S|_{\H_0(n)}$ and $S|_{\H_1(n)}$ must be equal since they are respectively first line of $\HZ_0$ and $\HZ_1$. 
Now take a pattern $w=(w_1,\dots,w_\ell)$ in $E_\kappa^\ell$, for some $\ell\leq |E_\kappa|$, and consider the word $W$ obtained by $j$ concatenations of $w$. The probability that $S|_{\H_0(j\ell)}$ and $S|_{\H_1(j\ell)}$ take value $W$,  are both equal to
 \[\frac{\l(\prod_{i=0}^{\ell-1} \IT{(D U)}{w_i}{w_{i+1 \mo \ell}}\r)^j}{\Trace((UD)^{\ell j})} = \frac{\l(\prod_{i=0}^{\ell-1} \IT{(U D)}{w_i}{w_{i+1\mo \ell}}\r)^j}{\Trace((UD)^{\ell j})},\]
where the denominators are equal. Therefore, we deduce \C \ref{cond:commut-cyl}.\par
Assume that \C \ref{cond:DUT} and \C \ref{cond:commut-cyl}  hold true for $D$ and $U$ some MK. Assume that $S$ is a  $(D,U)$-HZCMC on $\HZ_0$. Again $S|_{\H_0(n)}$ and $S|_{\H_1(n)}$ are respectively $DU$ and $UD$-CMC on $\cyl{n}$. By \C~\ref{cond:DUT} one sees that $S|_{\H_1(n)}$ is obtained from $S|_{\H_0(n)}$ by the PCA $A$. Let us see why \C \ref{cond:commut-cyl} implies that $S|_{\H_0(n)}$ and $S|_{\H_1(n)}$ have the same law: we have to prove that any word $W=(w_0,\dots,w_{n-1})$ occurs equally likely for $S|_{\H_0(n)}$ or $S|_{\H_1(n)}$, when \C \ref{cond:commut-cyl} says that it is the case only when $n\leq |E_\kappa|$. We will establish that  
\[\prod_{i=0}^{n-1} (UD)_{w_i,w_{i+1\mo n}}= \prod_{i=0}^{n-1} (DU)_{w_i,w_{i+1\mo n}}.\]
For any letter $a \in E_\kappa$ which occurs at successive positions $j_1^a,\dots,j_{k_a}^a$ for some $k_a$ in $W$ let $d_n^a(j_i^a,j_{i+1}^a)$ be the distance between these indices in $\cyl{n}$ that is $\min(j_{i+1}^a-j_i^a,n-j_{i+1}^a+j_i^a)$. Since $|E_\kappa|<+\infty$ is bounded, there exists $a$ and indices $j_i^a$ and $j_{i+1}^a$ for which $d_n(j_i^a,j_{i+1}^a)\leq |E_{\kappa}|$ (by the so called pigeonhole principle); to show that $W$ occurs equally likely in $S|_{\H_1(n)}$ and in $S|_{\H_0(n)}$ it suffices to establish that $W'$ obtained by removing the cyclic-pattern $W'=w_{j_i^a+1},\dots,w_{j_{i+1}^a}$ from $W$ occurs equally likely in  $S|_{\H_1(n-(j_{i+1}^a-j_{i}^a))}$ and $S|_{\H_0(n)-(j_{i+1}^a-j_{i}^a)}$ (since the contribution to the weight of the cyclic-pattern $W'$ 
is $\prod_{\ell=j_i^a}^{j_{i+1}^a-1} (DU)_{w_\ell',w_{\ell+1}'}=\prod_{\ell=j_i^a}^{j_{i+1}^a-1} (UD)_{w_\ell',w_{\ell+1}'}$ in both $S|_{\H_1(n)}$ and $S|_{\H_0(n)}$). This ends the proof by induction.
\end{proof}
 
Recall the definitions of $U^\eta, D^\eta, \gamma$ given before Theorem \ref{theo:HZcolors}.
\begin{theo}\label{theo:HZcolors-2}  
 Let $A:=(\cyl{n}, E_\kappa,N,T)$ be a positive rate PCA seen as acting on ${\cal M}(E_\kappa^{\HZ(n)})$. $A$ admits the law of a HZCMC as invariant distribution on $\HZ(n)$ iff \C \ref{cond:gibbs-1} holds and if\\~\\
\begin{cond}\label{cond:commut-cyl-2}
$ \displaystyle
    {\sf Diagonal}((D^\gamma U^\gamma)^k)={\sf Diagonal}((U^\gamma D^\gamma)^k) \textrm{~for all } 1\leq k\leq \kappa+1.
$
\end{cond}  ~\\~\\
In this case the $(D^\gamma,U^\gamma)$-HZCMC is invariant under $A$. 
When $\kappa=1$ (the two-colour case), when  \C \ref{cond:gibbs-1} holds then so does \C \ref{cond:commut-cyl-2},
 and then the only condition is \C \ref{cond:gibbs-1}.
\end{theo}
Again, one can state a version of this Theorem without the positive rate condition with \C \ref{cond:tauxg} instead (the analogous of Theorem \ref{theo:extension} in the cyclic case). The proof in this case is the same as that of Theorem \ref{theo:extension}.

\subsection{Markov invariant distribution on the line}
In this section, we discuss some necessary and sufficient conditions on $(M,T)$ for the law of the $M$-MC to be invariant under $T$ on $\H$ and $\H(n)$.

\subsubsection{Markovian invariant distribution on  $\H$ or $\H(n)$}
\label{sec:MID}
Let $T$ be a TM for a PCA $A$ in $\PCA(\LL,E_\kappa)$. 
Let $M$ be a MK on $E_\kappa$,  and $\rho = \bma \rho_i  \ema_{0 \leq i \leq \kappa}$an element of ${\cal M}^\star(E_\kappa)$.
Consider the matrices $(Q_x^M, x \in E_\kappa)$ defined by
\[Q_x^M=\bma \dis\frac{\sqrt{\rho_i}}{\sqrt{\rho_j}} M_{i,j}\TT{i}jx \ema_{0\leq i,j\leq \kappa},
\]
and set  $\Sr:=\bma \sqrt{\rho_i} \ema_{0\leq i\leq \kappa}$ (we should write $Q_x^M(\rho,T)$ instead, but $\rho$ and $T$ will be implicit). 

\begin{lem} \label{lem:ecr-mat} Let  $T$ be a TM for a PCA $A$ in $\PCA(\LL,E_\kappa)$ (with positive rate or not).
  \begin{enumerate}[(i)]
  \item The law of the  $(\rho,M)$-MC is invariant by $T$ on $\H$ iff for any $m>0$, any $x_1,\dots,x_m \in E_\kappa$, 
    \ben\label{eq:qgdsf}
\rho_{x_1}\prod_{j=1}^{m-1} M_{x_j,x_{j+1}} = \Sr \l(\prod_{j=1}^m Q^M_{x_j}\r) \tran{\Sr}.
\een
  \item The law of the  $M$-CMC is invariant by $T$ on $\H(n)$ iff for any $x_1,\dots,x_n \in E_\kappa$, 
    \ben \label{eq:phted}
   \prod_{j=1}^n M_{x_j,x_{j+1 \mo n}}= \Trace \l(\prod_{j=1}^n Q^M_{x_j}\r) .
  \een
\end{enumerate}
\end{lem}
\begin{proof}
  Just expand the right hand side.
\end{proof}

In the rest of this section, $(i)$ and $(ii)$ 
will always refer to the corresponding item in Lemma \ref{lem:ecr-mat}.
We were not able to fully describe the set of solutions $(M,T)$ to $(i)$ and $(ii)$.
Nevertheless, in the rest of this section we discuss various necessary and sufficient conditions on $(M,T)$. 
We hope that the following results will shed some light on the algebraic difficulties that arise here.

\begin{pro}\label{pro:PS}[I.I. Piatetski-Shapiro] Lemma \ref{lem:ecr-mat} still holds if in $(i)$ 
 the conditions \eref{eq:qgdsf}  
holds only for all $m\leq \kappa+2$.
\end{pro}
\begin{proof} We borrow the argument from Toom \& al. \cite[Theorem 16.3]{dkt}. First, note that Formula \eref{eq:qgdsf} can be rewritten
\ben\label{eq:crit}
\l(\frac{\rho^{1/2}}{\rho_{x_1}} Q^M_{x_1}-\frac{M_{x_1,x_2}}{\rho_{x_2}}\rho^{1/2}\r)\l(Q^M_{x_2}Q^M_{x_3}...Q^M_{x_m}\r)\,{}^t \rho^{1/2}=0. \een
We want to prove that if \eref{eq:crit}  holds for all $m\leq \kappa +2$ (and all $(x_i)'s$) then it holds also for any $m>\kappa+2$. The argument relies on the dimension of a certain associated vector space. 
Consider ${\cal P}^d$ the set of  monomials $P(Q^M_0,\dots,Q^M_\kappa)$ with degree at most $d$, that is an ordered product with at most $d$ terms (with possible repetitions) of some $Q^M_i$'s.

If \eref{eq:crit} holds for any $m\leq \kappa +2$ and some fixed $x_1,x_2$, then the row vector 
\ben\label{eq:sqdd}
v=\frac{\rho^{1/2}}{\rho_{x_1}} Q^M_{x_1}-\frac{M_{x_1,x_2}}{\rho_{x_2}}\rho^{1/2}
\een has the following property: for any $d\leq \kappa+1$ and  $P \in {\cal P}^d$, we have $v P(Q^M_0,\dots,Q^M_\kappa) \,{}^t \rho^{1/2}=0$.
This property can be rephrased as follows: all the vectors in the set $S=\{v P(Q^M_0,\dots,Q^M_\kappa), P \in {\cal P}^d\}$ are orthogonal to $\,{}^t \rho^{1/2}$, or equivalently, ${\sf Vect}(S)$ is orthogonal to  $\,{}^t \rho^{1/2}$.

Take now any vector $c$, and consider the vector spaces 
$L_{1}(c)={\sf Vect} (c)$ and for any $m\geq 1$,
\[L_{m+1}(c)= {\sf Vect}(L_m(c), \{ x Q^M_y, x\in L_m(c),  0\leq y \leq \kappa\}).\]
The sequence $L_m(c)$ is strictly increasing till it becomes constant, because its dimension is bounded by that of the ambient space $\kappa+1$. For this reason, it reaches its final size for some $m\leq \kappa+1$.
Hence, if the vector space $L_{\kappa+1}(c)$ is orthogonal to $\,{}^t \rho^{1/2}$, then so does the $L_m(c)$ for $m> \kappa+1$.\\
To end the proof, it remains to note that the polynomial which appears in \eref{eq:crit} has degree $m-1$.
\end{proof}

Since the asymptotics of $\Trace(A^n)$ or $\Sr A^n~ {}^t\Sr$ are driven by the largest eigenvalues of $A$ (under mild conditions on $(\Sr,A)$), we have the following statement which can be used as some necessary conditions on the system $(M,T)$.
 \begin{pro} \label{pro:MV-ML}
  \begin{enumerate}[(a)]
  \item Assume that $(M,T)$ is solution to $(i)$  with $T$ a positive rate TM, then for any $\ell\geq 1$, any $x_1,\dots,x_\ell$ we have $\prod_{i=1}^{\ell} M_{x_i,x_{i+1 \mo \ell}}= \MV{\prod_{j=1}^{\ell}Q^{M}_{x_j}}$ (recall that  $\ME(A)$ is the maximum eigenvalue of the matrix $A$).
   \item Let $\ell\geq 1$ be fixed. Assume that $(M,T)$ is solution to $(ii)$ 
for at least $\kappa +1$ (this is $|E_\kappa|$) different positive integers $n$ of the form $n=k\ell$. In this case, for any $x_1,\dots,x_\ell,$ $\MV{\prod_{j=1}^{\ell}Q^{M}_{x_j}}=\prod_{i=1}^{\ell} M_{x_i,x_{i+1 \mo \ell}}$.
Moreover, all the matrices $\prod_{j=1}^{\ell}Q^{M}_{x_j}$ have rank 1. 
 \end{enumerate}
\end{pro}

\begin{rem} In Proposition \ref{pro:MV-ML}, we can replace the positive rate condition by a weaker one: we only need the primitivity of the matrices $\prod_{j=1}^{\ell}Q^{M}_{x_j}$ for any $\ell, x_1,\dots,x_\ell$.  
  But this condition is a bit difficult to handle since it does not follow the primitivity of the family of matrices $Q_x^M$. 
\end{rem}

\noindent \it Proof. \rm We give a proof in the case $\ell=1$ and for case $(i)$ and $(ii)$ for sake of simplicity, but exactly the same argument applies for larger $\ell$ (by repeating the pattern $(x_1,\dots,x_\ell)$ instead of $x$ alone).
 Following Remark \ref{rem:cas_taux_positif}, the positive rate condition on $T$ implies that if the law of MC with MK $M$ is invariant by $T$, then the matrices $Q_{x_1}^M\dots Q_{x_\ell}^M$ have positive coefficients.
  \begin{enumerate}[(a)] 
  \item Let $m\geq 1$. Taking $x_1=\dots=x_m=x$ in Lemma~\ref{lem:ecr-mat}, we get
          $\rho_{x} M_{x,x}^{m-1} = \Sr \l(Q^M_x\r)^m \tran{\Sr}$.
        By Perron-Frobenius we obtain for $\RV{Q^M_x}$ and $\LV{Q^M_x}$ the Perron-RE and LE of $Q_x^M$ normalised so that $\LV{Q^M_x}\RV{Q^M_x}=1$,
    $\rho_{x} M_{x,x}^{m-1} \underset{m \to \infty}{\sim} \MV{Q^M_x}^m \l( \Sr \RV{Q^M_x} \LV{Q^M_x} \tran{\Sr} \r)$.
        Hence, necessarily, $\MV{Q^M_x} = M_{x,x}$.

  \item Let $x$ be fixed. Assume that  $(ii)$ holds for $\kappa+1$ different integers $n=n_i$ for $i=0,\dots,\kappa$. For all $n \in \{n_0,\dots,n_{\kappa}\}$, 
   $      M_{x,x}^n = \Trace \l(\l(Q^M_x\r)^n\r) = \sum_i \lambda_i^n $   where $(\lambda_i, 0\leq i\leq \kappa)$ are the eigenvalues of $Q^M_x$ from what we deduce that all the eigenvalues of $Q^M_x$ equals 0, but 1 which is $M_{x,x}$.~~$\Box$
  \end{enumerate}

One can design various sufficient conditions for $T$ to satisfy $(i)$, $(ii)$.   
For example,  for the case $(i)$  
following the proof of Proposition \ref{pro:PS}, it suffices that for any $x_1,x_2,$ 
$\frac{\rho^{1/2}}{\rho_{x_1}} Q_{x_1}^M=\frac{M_{x_1,x_2}}{\rho_{x_2}}\rho^{1/2}$
for the law of the  $M$-MC to be invariant under $T$.

\subsubsection{I.i.d. case} \label{sec:iid}
If we search to determine the TM $T$ for which there exists an invariant product measure (instead of more general MC), the content of Section \ref{sec:MID} still applies since product measures are MC whose MK satisfies, for any $(a,b)\in E_\kappa^2$,  $M_{a,b}=\rho_b$. In this case 
\[Q_x^M=Q^\rho_x=\bma\sqrt{\rho_i}\,\, \TT{i}jx\,\sqrt{\rho_j}  \ema_{0\leq i,j\leq \kappa}.\]
The iid case is also interesting, as has been shown by Mairesse \& Marcovici \cite{1207.5917}.
We can design some additional sufficient conditions for the product measure $\rho^{\mathbb{Z}}$ to be invariant under $T$. For example if
\begin{equation} \label{eq:MM}
\frac{\rho^{1/2}}{\rho_{x}} Q_{x}^\rho = \rho^{1/2}, ~~\textrm{ for any } x\in E_{\kappa}
\end{equation} (see also  Mairesse \& Marcovici \cite[Theorem 5.6]{1207.5917}), or if for any  words $W_1$ and $W_2$, and any $0\leq x,y \leq \kappa$, $\dis Q^\rho_{W_1} (Q^\rho_xQ^\rho_y-Q^\rho_yQ^\rho_x)Q^\rho_{W_2}=0$ then  $\rho^{\mathbb{Z}}$ is invariant under $T$. Necessary and sufficient conditions on $T$ seem out of reach for the moment.

\subsubsection{Symmetric transition matrices and i.i.d. invariant measure }
\label{sec:stiim}

We say that a TM $T$ is symmetric if for any $a,b,c \leq \kappa$, $\TT{a}bc=\TT{b}ac$.
Let $T$ be a symmetric transition matrix of a PCA $A$ in $\PCA(\LL,E_\kappa)^\star$ with positive rate
and let $\rho \in {\cal M}^\star(E_\kappa)$ be a distribution on $E_\kappa$ with full support. 
A distribution $\mu$  in ${\cal M}(E_\kappa^\mathbb{Z})$ is said to be symmetric if $\mu(x_1,\dots,x_n)=\mu(x_n,\dots,x_1)$ for any $n\geq 1, 0\leq x_1,\dots, x_n \leq \kappa$.
We start by two simple observations valid for PCA with a symmetric TM:\\ 
$\bullet$ by a compacity argument (easily adapted from Prop. 2.5 p.25 in \cite{dkt}), there exists a symmetric distribution $\mu$ in ${\cal M}(E_\kappa^\mathbb{Z})$ invariant by $T$. \\
$\bullet$ for any $x$ the matrix $Q^\rho_x$ is symmetric and then Hermitian. Hermitian matrices possess some important properties, which help to go further:
\begin{enumerate}[(a)]
\item $r$ is a right eigenvector for an Hermitian matrix $A$ associated with the eigenvalue $\lambda$ (that is $rA=\lambda A$) iff ${}^t r$ is a right eigenvector of $A$ associated with $\lambda$ (that is $A{}^t r=\lambda{}^t r$).
\item If  $A$ and $B$ are two Hermitian matrices then $\MV{A+B} \leq \MV{A} + \MV{B}$. The equality holds only if the (left, and then right by $(a)$) eigenspaces of the matrices $A$ and $B$ associated with the respective eigenvalues $\MV{A}$ and $\MV{B}$ are equal.
\end{enumerate}
 
\begin{pro} \label{pro:rhocom} Let $T$ be a symmetric TM with positive rate.
  \begin{enumerate}[$(a)$]
  \item   $\rho^{\Z}$ is invariant by $T$ on $\H$ iff \eref{eq:MM} holds.
  \item   $\rho^{\Z(n)}$ is invariant by $T$ on $\H(n)$ for at least $\kappa+1$ different positive integers $n$ iff for any $i,j,x \in E_\kappa$, $\TT{i}jx = \rho_x$.
\end{enumerate}
\end{pro}
The positive rate condition allows one to use Perron-Frobenius theorem on the matrices $(Q_x^\rho,x \in E_\kappa)$  in the case where  the Perron-eigenspaces have dimension 1. The proposition still holds if we replace the positive rate condition by a weaker one for example the primitivity of the matrices $(Q_x^\rho,x \in E_\kappa)$.\\
\it Proof. \rm $(a)$  Assume first that $\rho^{\Z}$ is invariant by $T$ on $\H$. Then, we have, for any $n\geq 1$, the equality $\Sr \l( Q^\rho_{x}\r)^n \tran{\Sr}=\rho_x^n$. By Frobenius, we deduce that $\ME(Q^\rho_x)=\rho_x$. Hence, $\MV{\sum_x Q_x^\rho}= \sum_x\MV{Q_x^\rho}=1$ (by the properties of Hermitian matrices recalled above), all the matrices $Q_x^\rho$ and $\sum_x Q_x^\rho$ have same Perron-LE and RE that are $\Sr$ and $\tran{\Sr}$. \par
Reciprocally, assume that for all $x \in E_\kappa$, the Perron eigenvalue of $Q_x^\rho$ is $\rho_x$ and  $\Sr$ and $^t{} \Sr$ are Perron-LE and RE of $Q_x^\rho$. Then for any $m$ for any $x_1,\dots,x_m \in E_\kappa$,
\begin{displaymath}
\Sr \l(\prod_{j=1}^m Q^\rho_{x_j}\r) \tran{\Sr} = \rho_{x_1} \Sr \l(\prod_{j=2}^{m} Q^\rho_{x_j}\r) \tran{\Sr} = \dots = \prod_{j=1}^m \rho_{x_j} \Sr ~\tran{\Sr} = \prod_{j=1}^m \rho_{x_j}
\end{displaymath}
which means that $\rho^\mathbb{Z}$ is invariant by $T$.\\
Proof of $(b)$.  By the same argument as in $(a)$, $Q_x^\rho$ and  $\sum_x Q_x^\rho$ have $\Sr$ and $\tran{\Sr}$ for Perron-LE and RE. Moreover since the rank of $Q_x^\rho$ is $1$ (see Proposition~\ref{pro:MV-ML}),
\begin{displaymath}
Q_x^\rho = \MV{Q_x^\rho} \RV{Q_x^\rho} \LV{Q_x^\rho} = \rho_x ~{}^t{\Sr} \Sr = \bma ~{}^t\Sr_i \rho_x  \Sr_j \ema_{i,j}.
\end{displaymath}
But, $\l(Q_x^\rho\r)_{i,j} = ~{}^t\Sr_i \TT{i}jx \Sr_j$. Then, for all $x,i,j$, $\TT{i}jx = \rho_x$. (The converse is trivial)~~~$\Box$

\section{Markov invariant distribution on $\H$ vs $\H(n)$ vs  $\HZ$ vs $\HZ(n)$}
\label{sec:RM}
Consider a TM $T$ seen as acting on $\H$, $\H(n)$, $\HZ$ and $\HZ(n)$. In this section we discuss the different conclusions we can draw from the Markovianity of the invariant distribution under $T$ on one of these structures. 
Before going into the details, we think that it is interesting to note that any Markov measure on $\H$ is the invariant measure for a PCA (as stated Prop. 16.1 in \cite{dkt}).

Figure \ref{fig:imp} gathers most of the results obtained in this section. 
\begin{figure}[ht]
\centerline{\includegraphics{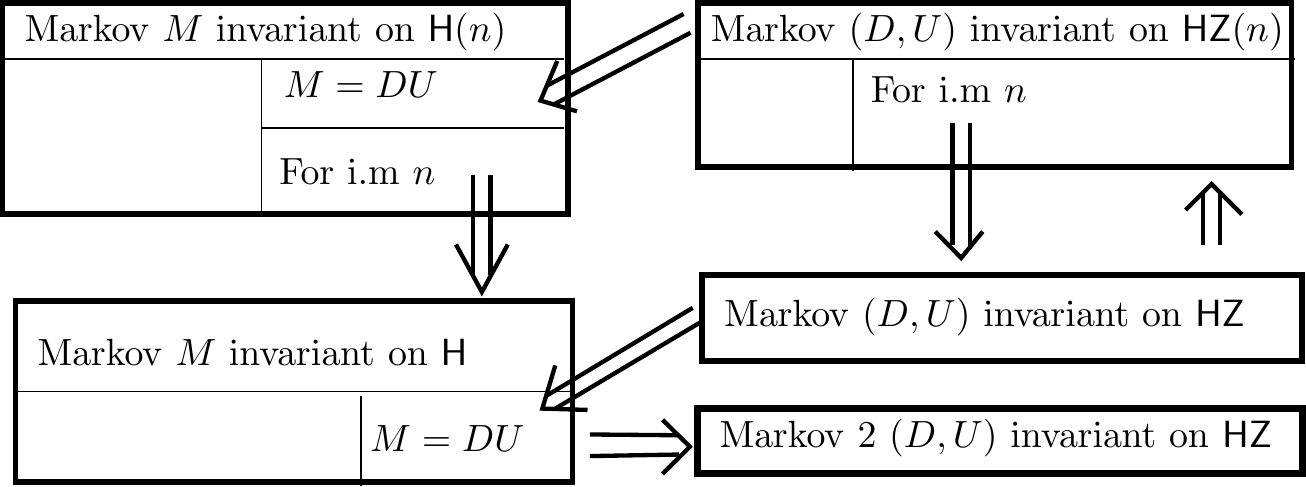}}
\captionn{\label{fig:imp} Relations between the existence of Markovian invariant distribution on the different structures. ``i.m.'' means ``infinitely many''}
\end{figure}

\paragraph{From $\H(n)$ to $\H$.} The following Proposition is already known (see Albenque \cite{Alb} and a ``formal'' version is also used in Bousquet-Mélou \cite{MBM}).
\begin{pro}\label{pro:LntoL}
If a \PCA{}  $A:=(\cyl{n}, E_\kappa,N,T)$ admits the  law of a CMC with an irreducible MK $M$ on $\H(n)$ as an invariant distribution for infinitely many $n$ then the PCA $A:=(\Z, E_\kappa,N,T)$ admits the law of the  $M$-MC as an invariant distribution on $\H$.
\end{pro}
\noindent\it Proof.  \rm  Several proofs are possible. We adapt slightly the argument of Theorem~3 in~\cite{Alb}. The idea is to prove that the law of a $M$-CMC on $\H(n)$ converges to a $M$-MC on the line (limit taken in the set $O$, the set of integers $n$ for which the law of the  $M$-MC is invariant by $T$ on $H(n)$). Proceed as follows. Choose some $k\geq 1$. For $n\geq k$   in $O$, the probability of any pattern $b_1,\dots,b_k$  in $E_\kappa$ (in successive positions)  is for this distribution
\beq\label{eq:fond}
\l(\prod_{i=1}^{k-1} M_{b_i,b_{i+1}}\r) (M^{n-k})_{b_k,b_1}=\sum_{(a_1,\dots,a_{k+1})\in E_\kappa^{k+1}} \l(\prod_{i=1}^{k} M_{a_i,a_{i+1}} \TT{a_i}{a_{i+1}}{b_i}\r) (M^{n-k-1})_{a_{k+1},a_1}.
\eq
Since $M$ is an irreducible and aperiodic MK, by Perron-Frobenius theorem, $M^n\to M^{\infty}$ where $M^{\infty}$ is the matrix whose rows equal the stochastic LE $\rho$ of $M$. Therefore $(M^{n-k})_{b_k,b_1}\to \rho_{b_1}$ and the limit distribution for $\H(n)$ exists and satisfies
\[\P(S_i=b_i,i=1,\dots,k)=\rho_{b_1}\prod_{i=1}^{k-1} M_{b_i,b_{i+1}}\]
and satisfies, taking the limit in \eref{eq:fond},
\beq\label{eq:fond2}
\rho_{b_1}\prod_{i=1}^{k-1} M_{b_i,b_{i+1}}=\sum_{(a_1,\dots,a_{k+1})\in E_\kappa^{k+1}} \rho_{a_1}\l(\prod_{i=1}^{k} M_{a_i,a_{i+1}}\TT{a_i}{a_{i+1}}{b_i}\r).~~~~ \Box
\eq

\paragraph{From $\H$ to $\HZ$.}
\begin{pro}If the law of a $M$-MC is an invariant distribution for a PCA $A:=(\mathbb{Z}, E_\kappa,N,T)$ on the line, then seen as acting on the set of measures indexed by $\HZ$, $A$ admits the law of a HZMC with memory 2 as invariant distribution. 
\end{pro}
\proof Take $D_{a,c}=\sum_i M_{a,i} \TT{a}{i}{c}$ and $U_{a,c,b}= \frac{M_{a,b}\TT{a}bc}{\sum_i M_{a,i}\TT{a}ic}$ or 0 if the denominator is 0 (in which case the numerator is 0 too). These kernels have to be understood as follows:
\[\P(S(0,1)=c | S(0,0)=a)=D_{a,c},~~ \P(S(1,0)=b ~|~S(0,0)=a, S(0,1)=c)=U_{a,c,b},\]
and they satisfy 
\ben\label{eq:M2}
D_{a,c}U_{a,c,b}=M_{a,b}\TT{a}bc.
\een
Roughly the Markov 2 property along the zigzag is Markov 1 along a $D$ steps and Markov 2 along a $U$ step. Now if the law of a  $M$-MC is invariant on $\H$, then for $\rho$ stochastic LE of $M$, we have by \eref{eq:M2}
\be
\P(S(i,0)=a_i,i=0,\dots,n+1, S(i,1)=b_i, i \in 0,\dots,n)&=&\rho_{a_0}\prod_{i=0}^n M_{a_i,a_{i+1}}  \TT{a_i}{a_{i+1}}{b_i}\\
&=&\rho_{a_0}\prod_{i=0}^n D_{a_i,b_{i}} U_{a_i,b_i,a_{i+1}}
\ee
which is indeed the representation of a Markov 2 process with MK $(D,U)$ on $\HZ$.~$\Box$

\begin{rem}In the previous proof we saw that if $M$ is Markov on $\H$, then it is Markov 2 on $\HZ$ with memory 1 on a down step, and 2 on a up step. What it is true too, is that to this kind of process one can associate a Markov 1 process with MK $M'$ on $\H$ with values in $E_\kappa^2$ (as illustrated on Figure \ref{fig:reck}) by ``putting together'' the state $S_t(i)$ and $S_{t+1}(i)$. \begin{figure}[ht]
\centerline{\includegraphics{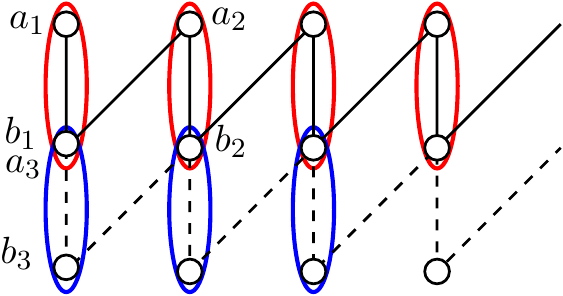}}
\captionn{\label{fig:reck} From PCA with Markov 2 invariant distribution to PCA with Markov 1.}
\end{figure}
The associated PCA is $A'=(\Z,E_\kappa^2,N,T')$ with 
$\TTp{(a_1,b_1)}{(a_2,b_2)}{(a_3,b_3)}=1_{b_1=a_3}\TT{b_1}{b_2}{b_3}$ and the MK is 
\[M'_{(a_1,b_1),(a_2,b_2)}=U_{a_1,b_1,a_2}D_{a_2,b_2}.\] 
Nevertheless the PCA $A'$ has a lot of transitions equal to 0 which makes that our criterion for Markovianity fails.
\end{rem}

\paragraph{From $\HZ(n)$ to $\H(n)$ and from $\HZ$ to $\H$.}
We have already said that the restrictions of a HZMC on $\HZ_t$ (resp. a HZCMC on $\HZ_t(n)$) on the lines on $\H_t$ and $\H_{t+1}$ (resp. $\H_t(n)$ and $\H_{t+1}(n)$) were MC (resp. CMC). As a consequence,  if a PCA  $A:=(\LL, E_\kappa,N,T)$ seen as acting on ${\cal M}(E_\kappa^{\HZ})$ (resp. ${\cal M}(E_\kappa^{\HZ(n)})$) admits the law of a HZMC (resp. HZCMC) as invariant distribution, then seen as acting on ${\cal M}(E_\kappa^{\H})$ (resp. ${\cal M}(E_\kappa^{\H(n)})$), it admits the law of a MC (resp. CMC) as invariant distribution. \par

 \begin{rem} \label{rem:HvsHZ}$\bullet$ The converse is not true. Indeed as seen in Section \ref{sec:KR} if $\TT110 \TT001 = \TT100 \TT011$ or $\TT110 \TT001 = \TT010 \TT101$, a product measure is invariant on $\H$ but one can check that in these cases the stationary distribution on $\HZ$ is not a HZMC.
\end{rem}

\paragraph{From $\HZ$ to $\HZ(n)$.}
\begin{pro} Let $A:=(\mathbb{Z}, E_\kappa,N,T)$ be a PCA. If the law of the  $(D,U)$-HZMC on $\HZ$ is invariant by $A$ then 
the  law of the  $(D,U)$-HZMC on $\HZ(n)$ is invariant by $A$.
\end{pro}
\Proof Just compare the hypothesis of Theorems \ref{thm:clzz} and \ref{thm:cczz}.$~\Box$

\paragraph{From $\H$ to $\H(n)$.} 
In the case $\kappa=1$, there exists some PCA that have a product measure invariant on $\H$ that are not Markov on $\H(n)$. To be invariant on $\H(n)$ for  infinitely many $n$ implies that the matrices $(Q_x^\rho, x \in \{0,1\})$ have rank 1 (Proposition \ref{pro:MV-ML} $(b)$).
In Section \ref{sec:KR} we have seen that, when $\TT101 \TT010 = \TT110 \TT001$, a product measure was invariant on $\H$. The computation in this case (in the positive rate case) gives $\rho_0= \l(\frac{1-\TT111}{\TT{1}01+\TT{0}10}\r)$; and with this value one checks that neither $Q_0^\rho$ nor $Q_1^\rho$ have rank 1. This does not prove the non existence of a product measure depending on $n$, invariant by the PCA acting on $\H(n)$.

\section{Proofs of Theorems \ref{theo:HZcolors},  \ref{theo:HZcolors-2} and \ref{theo:extension}}
\label{sec:PTHZ}
 
We prove Theorem \ref{theo:extension} at the end of the section. \par
To prove Theorem \ref{theo:HZcolors} and \ref{theo:HZcolors-2} we will use the characterisation given by Theorem \ref{thm:clzz} (the proof of Theorem \ref{theo:HZcolors-2} is similar, see Remark \ref{rem:theo:HZcolors-2}). First we will show the crucial following Lemma, a cornerstone of the paper.

\begin{lem}\label{lem:C1C2} Let $T$ be a positive rate TM.
The two conditions \C \ref{cond:gibbs-g} and \C \ref{cond:gibbs-1} are equivalent. They are also equivalent to the existence of a  pair of MK $(D,U)$ satisfying \C \ref{cond:DUT}.
\end{lem}
\begin{proof} $\bullet$ Assume first that there exists $(D,U)$ satisfying \C \ref{cond:DUT}, and let us see that \C \ref{cond:gibbs-g} is satisfied: substitute $\TT{i}jk$ by their expression in terms of $(D,U)$ as specified in \C \ref{cond:DUT} in the equation defining  \C \ref{cond:gibbs-g} , and check that both sides are equal. \\
 $\bullet$ Proof of \C \ref{cond:gibbs-g}$\imp$ \C \ref{cond:gibbs-1}: take $a'=b'=c'=0$ in \C \ref{cond:gibbs-g}.\\
 $\bullet$ Proof that \C \ref{cond:gibbs-1} implies  the existence of $(D,U)$ satisfying \C \ref{cond:DUT}.  Suppose \C \ref{cond:gibbs-1}  holds and let us find $D$ and $U$ such that 
\ben\label{eq:mas}
\frac{D_{a,b}U_{b,a'}}{ (DU)_{a,a'}}=\TT{a}{a'}{b}, \textrm{ for any } a,b,a'.
\een
It suffices to find $D$ and $U$ such that
\ben\label{eq:DUT}
D_{a,b}U_{b,a'}=\frac{\TT{a}0b \TT{0}{a'}{b}}{\TT00b} G[a,a']
\een
for some numbers $(G[i,j],0\leq i,j \leq \kappa)$, since in this case
\be
\frac{D_{a,b}U_{b,a'}}{ (DU)_{a,a'}}&=& \frac{\frac{\TT{a}0b \TT{0}{a'}{b}}{\TT00b} G[a,a']}{\sum_i \frac{\TT{a}0i \TT{0}{a'}{i}}{\TT00i} G[a,a']}
=\frac{\frac{\TT{a}0b \TT{0}{a'}{b}}{\TT00b}  }{ \frac{\TT{a}00 \TT{0}{a'}{0}}{\TT000\TT{a}{a'}{0}}  }=\TT{a}{a'}{b}.
\ee
Now, a solution to \eref{eq:DUT} is given by 
\ben\label{eq:DUG}
D_{a,b}= \frac{\TT{a}0b}{\TT{0}{0}{b}} A_{a}B_{b},~~ U_{b,a'}= {\TT{0}{a'}{b}} \frac{C_{a'}}{B_b},  ~~ G[a,a']=A_aC_{a'}
\een
where $C=(C_a, 0\leq a \leq \kappa)$ is any array of positive numbers, $B=(B_a, 0\leq a \leq \kappa)$ is chosen such that $U$ is a MK, and then $A=(A_a, 0\leq a \leq \kappa)$ such that $D$ is a MK.
\end{proof}
We now characterise the set of solutions $(D,U)$ to \C \ref{cond:DUT} when $T$ satisfies \C \ref{cond:gibbs-1}.
 \begin{pro}  \label{pro:zea}
Let $T$ with positive rate satisfying \C \ref{cond:gibbs-1}.
The set of pairs $(D,U)$ solutions to \C \ref{cond:DUT} is the set of pairs $\{(D^\eta,U^\eta), \eta \in {\cal M}^\star(E_\kappa)\}$ (indexed by the set of distributions $\eta=(\eta_a, 0\leq a \leq \kappa)$ with full support)
as defined in \eref{eq:iqb}. 
\end{pro}
\begin{proof} Assume that \C \ref{cond:gibbs-1} holds. By Lemma \ref{lem:C1C2}, there exists $(D,U)$ satisfying \C \ref{cond:DUT}, that is such that for any  $0\leq a,b,c\leq \kappa$, $\TT{a}{b}{c}(DU)_{a,b}   = D_{a,c}U_{c,b}$. 
If all the $\TT{a}{b}{c}$ are positive, then for any $a,b$, $D_{a,b}$ and $U_{a,b}$ are also positive, and then one gets
\ben\label{eq:DUT222}
D_{a,c}U_{c,b}= \frac{D_{a,0}U_{0,b}}{\TT{a}{b}{0}}\TT{a}{b}{c}
\een
and then
\ben\label{eq:DUT222b}
 (DU)_{a,b}= \frac{D_{a,0} U_{0,b}}{\TT{a}{b}{0}}.
\een
In the positive rate case, it is also true that if $(D,U,T)$ satisfies \eref{eq:DUT222} and \eref{eq:DUT222b} then $(D,U)$ satisfies \C \ref{cond:DUT}.  Observe that \eref{eq:DUT222} implies (summing over $b$),
\ben\label{eq:DUT222-2}
D_{a,c}=D_{a,0}\sum_b \frac{U_{0,b}}{\TT{a}{b}{0}}\TT{a}{b}{c},
\een
and then by \eref{eq:DUT222} again (replacing $D_{a,c}$ by the right hand side of \eref{eq:DUT222-2}) we get
\ben\label{eq:Ucb}
U_{c,b}  &=& \frac{ \frac{U_{0,b}}{\TT{a}{b}{0}}\TT{a}{b}{c}}{ \sum_{b'} \frac{ U_{0,b'}}{\TT{a}{b'}{0}}\TT{a}{b'}{c}}.
\een
Notice at the right hand side, one can replace $a$ by $0$ since \C \ref{cond:gibbs-1} holds.

Now, summing over $c$ in \eref{eq:DUT222-2} we get 
\ben\label{eq:DUT2223}
D_{a,\star}=1=\sum_c\sum_b \frac{D_{a,0}U_{0,b}}{\TT{a}{b}{0}}\TT{a}{b}{c} =\sum_b \frac{D_{a,0}U_{0,b}}{\TT{a}{b}{0}} ,
\een
which implies 
\ben\label{eq:Da0}
D_{a,0}=\l({\sum_{b}U_{0,b}/\TT{a}{b}{0}}\r)^{-1}.
\een
We then see clearly that the distributions $\eta$ defined by
\[\eta_b=U_{0,b},~~b=0,\dots,\kappa\] 
can be used to parametrise the set of solutions.  Replacing $U_{0,b}$ by $\eta_b$ in \eref{eq:Da0}, we obtain that $D_{a,0}=\l({\sum_{b}\eta_b/\TT{a}{b}{0}}\r)^{-1}$. Now,  using this formula in \eref{eq:DUT222-2} and again the fact that $U_{0,b}=\eta_b$, we get the representation of $D^\eta$ as defined in \eref{eq:iqb}. The representation of $U^\eta$ (provided in \eref{eq:iqb}) is obtained by replacing $U_{0,b}$ by $\eta_b$ in   \eref{eq:Ucb}. \par
We have establish that $(D,U)$ satisfies \C \ref{cond:DUT} implies $(D,U)=(D^\eta,U^\eta)$ for $\eta=(U_{0,b},b \in E_\kappa)$.\par
 Reciprocally, take any distribution $\eta\in {\cal M}^\star(E_\kappa)$ and let us check that
\ben\label{eq:id}
\frac{D_{a,c}^{\eta}U_{c,b}^\eta}{\sum_{c'}D_{a,c'}^{\eta}U_{c',b}^\eta}  =\TT{a}bc
\een (the definition of $D^\eta$ and $U^\eta$ are given in \eref{eq:iqb}). It is convenient to start by noticing that under \C \ref{cond:gibbs-1} for any $a,b,c$,
 \[U_{c,b}^\eta=  \frac{ \frac{\eta_{b}}{\TT{0}{b}{0}}\TT{0}{b}{c}}{ \sum_{b'} \frac{ \eta_{b'}}{\TT{0}{b'}{0}}\TT{0}{b'}{c}} = \frac{ \frac{\eta_{b}}{\TT{a}{b}{0}}\TT{a}{b}{c}}{ \sum_{b'} \frac{ \eta_{b'}}{\TT{a}{b'}{0}}\TT{a}{b'}{c}}.\]
Thanks to this, one sees that
\[D_{a,c}^{\eta}U_{c,b}^\eta=
\frac{ 1}{\sum_{b'}\frac{\eta_{b'}}{\TT{a}{b'}{0}}}       \frac{ \frac{\eta_{b}}{\TT{a}{b}{0}}\TT{a}{b}{c}}{1 }\]
from which \eref{eq:id} follows readily.
\end{proof}

We end now the proof of Theorem  \ref{theo:HZcolors}. 
A consequence of the previous considerations is that there exists $(D,U)$ satisfying  \C \ref{cond:DUT} and \C \ref{cond:commut}
iff there exists $\eta$ such that  $D^\eta U^\eta =U^\eta D^\eta$, and of course, in this case $(D,U)= (D^\eta, U^\eta)$ satisfies \C \ref{cond:DUT}.   
Not much remains to be done: we need to determine the existence (or not) and the value of $\eta$ for which  $D^\eta U^\eta =U^\eta D^\eta$ and if such a $\eta$ exists compute the invariant distribution of the MC with MK $U^\eta$ and $D^\eta$. \par
We claim now that if  $D^\eta U^\eta =U^\eta D^\eta$ , then $\eta=\gamma$ the stochastic Perron-LE of $X$. As a consequence there exists at most one distribution $\eta$ such that $D^\eta U^\eta =U^\eta D^\eta$. To show this claim proceed as follows. 
Assume that there exists $\eta$ such that $D^\eta U^\eta =U^\eta D^\eta$ (where $D^\eta U^\eta$ is given in \eref{eq:DU}, and  $U^\eta D^\eta$ is computed as usual, starting from \eref{eq:iqb}).
By \eref{eq:DUT222b} and \eref{eq:Da0}  we  have  
\ben\label{eq:DU}
(D^\eta U^\eta)_{a,b}= \frac{1}{\sum_{d}\frac{\eta_d}{\TT{a}{d}{0}}}  \frac{  \eta_{b}}{\TT{a}{b}{0}}.
\een
Hence $D^\eta U^\eta =U^\eta D^\eta$ is equivalent to  
\ben\label{eq:DU2}
  \frac{1}{\sum_{d}\frac{\eta_d}{\TT{a}{d}{0}}}  \frac{  \eta_{b}}{\TT{a}{b}{0}}=\sum_c
\frac{ \frac{\eta_{c}}{\TT{0}{c}{0}}\TT{0}{c}{a}}{ \sum_{b'} \frac{ \eta_{b'}}{\TT{0}{b'}{0}}\TT{0}{b'}{a}  } 
\frac{\sum_\ell \frac{  \eta_{\ell}}{\TT{c}{\ell}{0}}\TT{c}{\ell}{b} }{\sum_{b''}\frac{\eta_{b''}}{\TT{c}{b''}{0}}},\textrm{ for any } a,b
\een
Replace 
$\frac{\TT{c}{\ell}{b}}{\TT{c}{\ell}{0}}$ by $\l(\frac{\TT0\ell{b}}{\TT{0}{\ell}0} \r)\frac{\TT{0}00\TT{c}0b}{\TT{c}00\TT00b}$ and introduce 
\ben\label{eq:sddd}
g_a=\l({\sum_{d} {\eta_d}/{\TT{a}{d}{0}}}\r)^{-1},~~ f_a= \sum_{b'} \frac{ \eta_{b'}}{\TT{0}{b'}{0}}\TT{0}{b'}{a} 
\een
\eref{eq:DU2} rewrites
\ben\label{eq:DUdsfw}
 g_a \frac{  \eta_{b}}{\TT{a}{b}{0}}=
\sum_c
\frac{ \frac{\eta_{c}}{\TT{0}{c}{0}}\TT{0}{c}{a}}{f_a  } 
 f_b\frac{\TT{0}00\TT{c}0b}{\TT{c}00\TT00b}  g_c,\textrm{ for any } a,b, 
\een
and, using \C  \ref{cond:gibbs-1} again, $ \frac{1}{\TT{0}{c}{0}}\frac{\TT{0}00\TT{c}0b}{\TT{c}00\TT00b}= \frac{\TT{c}cb}{\TT{c}c0\TT0cb}$, \eref{eq:DUdsfw} is equivalent to
\ben\label{eq:DUdfw}
  \frac{g_a  \eta_{b}}{\TT{a}{b}{0}}=
\frac{f_b}{f_a}\sum_c
\frac{  g_c \eta_{c}   }{ \TT{c}c0 } \frac{\TT{0}{c}{a}}{\TT0cb} \TT{c}cb,\textrm{ for any } a,b.
\een
The question is still here to find/guess, for which TM $T$ there exists $\eta$ solving this system of equations (some $\eta's$ are also hidden in $g$ and $f$). In the sequel we establish that there exists at most one $\eta$ that solves the system: it is $\gamma$. For this we notice that for $a=b$ this system \eref{eq:DUdfw} simplifies: $(D^\eta U^\eta)_{a,a} =(U^\eta D^\eta)_{a,a}$ (for any $a$) is equivalent to 
\ben\label{eq:DUdssqfw}
  \frac{g_a  \eta_{a}}{\TT{a}{a}{0}}=
 \sum_c
\frac{  g_c \eta_{c}   }{ \TT{c}c0 }  \TT{c}ca, \textrm{~ for any }a
\een
which is equivalent to the matrix equation:
\ben\label{eq:vpvp}
\bma   \frac{  g_a\eta_{a}}{\TT{a}{a}{0}},a =0,\ldots,\kappa \ema=\lambda^\star \nu ,
\een
where $\nu$ is the stochastic Perron-LE of $Y$ and $\lambda^\star$ some free parameter. By \eref{eq:sddd}, \eref{eq:vpvp} rewrites
\[\frac{1}{\sum_{d}\frac{\eta_d}{\TT{a}{d}{0}}}\frac{  \eta_{a}}{\TT{a}{a}{0}}=\lambda^\star\nu_a\]
and taking the inverse, we see that $\eta$ needs to be solution to
\ben\label{eq:ofd}
\sum_{d}{\eta_d \frac{\TT{a}{a}{0}\nu_a}{\TT{a}{d}{0}}} =\frac{1}{\lambda^\star} {\eta_{a}}{}.
\een
The only possible $\eta$ is then $\gamma$ the unique Perron-LE of $X$ (which can be normalised to be stochastic), and we must have
\ben\label{eq:star} 1/\lambda^\star=\lambda,
\een the Perron-eigenvalue of $X$.  Hence $D^\eta U^\eta=U^\eta D^\eta$ implies $\eta=\gamma$. Nevertheless this does not imply  $D^\gamma U^\gamma=U^\gamma D^\gamma$ and then the condition $D^\gamma U^\gamma=U^\gamma D^\gamma$ remains in Theorem \ref{theo:HZcolors} (what is true in all cases is $(D^\gamma U^\gamma)_{a,a}=(U^\gamma D^\gamma)_{a,a}$ for any $a$). 

However when $\kappa=1$ this is sufficient since one can deduce the equality of two MK $K$ and $K'$ from  $K_{0,0}=K'_{0,0}$ and $K_{1,1}=K'_{1,1}$ only. This is why in Theorem \ref{theo:HZcolors} a slight simplification occurs for the case $\kappa=1$. 
When $\kappa>1$ this is no more sufficient. 
\begin{rem}\label{rem:theo:HZcolors-2}
This ends the proof of Theorem \ref{theo:HZcolors-2} since we see that \C \ref{cond:gibbs-1} and ${\sf Diagonal}(DU)={\sf Diagonal}(UD)$ imply $\eta=\gamma$ (the converse in Theorem \ref{theo:HZcolors-2} is easy). And the discussion just above the remark suffices to check the statement concerning the case $\kappa=1$.
\end{rem}
It remains to find the stochastic Perron-LE $\rho$ of  $D^\gamma U^\gamma$.

Consider \eref{eq:DU}, where $\eta$ is now replaced by $\gamma$. By \eref{eq:ofd} and \eref{eq:star}, we have
\[  (D^\gamma U^\gamma)_{a,b} = (1/\lambda)\frac{\TT{a}a0 \nu_a \gamma_b}{ \TT{a}b0 \gamma_a}.\]
Hence, $\rho$ is characterized as the vector whose entries sum to 1, and such that,
   \ben\label{eq:dqz} \sum_{a} \frac{\TT{a}a0 \nu_a \gamma_b}{\TT{a}b0 \gamma_a} \rho_a = \lambda \rho_b,~~ \textrm{ for any }b\in E_\kappa.\een
Taking   $\mu_i = \rho_i / \gamma_i$, \eref{eq:dqz} is equivalent to
\[   \sum_{a} \frac{\TT{a}a0 \nu_a}{\TT{a}b0} \mu_a = \lambda \mu_b,~~ \textrm{ for any }b\in E_\kappa\]
which means that $\mu$ is the Perron-RE of $X$. We have obtained that $\rho_i=\mu_i \gamma_i$.
Since $D^\gamma U^\gamma=U^\gamma D^\gamma$, the Perron-LE of $D^\gamma$ and $U^\gamma$ coincide with that of $D^\gamma U^\gamma$. 
This ends the proof of Theorem \ref{theo:HZcolors}.

\subsection{Proof of Theorem \ref{theo:extension}}

We follow the arguments of the proof of Theorem \ref{theo:HZcolors} and adapt them slightly to the present case. The only difference is that \C \ref{cond:tauxg} replaces the positive rate condition.

Lemma \ref{lem:C1C2} still holds if instead of the positive rate condition we take \C \ref{cond:tauxg} (Remark \ref{rem:dazq} is needed to see why  \C \ref{cond:gibbs-1} $\imp$ \C \ref{cond:gibbs-g}, and the positivity of $\TT{a}00$ and $\TT{0}b0$ to see that there exists $(D,U)$ satisfying moreover $(DU)_{a,b}>0$ for all $a,b$). Also, we have $D_{a,0}>0$, $U_{0,b}>0$, $D_{0,a}>0$ and $U_{0,b}>0$ for any $a,b$ by \C \ref{cond:DUT}.
In \eref{eq:iqb}, $D_{a,c}^{\eta}$ and $U_{c,b}^\eta$ are well defined under \C \ref{cond:tauxg} only. \eref{eq:DUT222} still holds for the same reason, and again the pair of conditions \eref{eq:DUT222} and \eref{eq:DUT222b} is equivalent to \C  \ref{cond:gibbs-1} under \C \ref{cond:tauxg} only. \eref{eq:DUT222-2} still holds, but there is a small problem for \eref{eq:Ucb} since the division by $D_{a,c}$ is not possible for all $a$. The $D_{a,c}$ (for fixed $c$) are not 0 for all $a$ since $D_{0,c}>0$. So \eref{eq:Ucb} holds for the $a$ such that $D_{a,c}>0$. If all the $\TT{a}{b}{c}=0$ then take $U_{c,b}^\gamma=0$. 
The rest of the proof of Theorem \ref{theo:HZcolors} can be adapted with no additional problem.~~$\Box$

\subsection*{Acknowledgements}

We thank both referees for their numerous remarks, suggestions and corrections, that really helped us to improve the paper

\small
\bibliographystyle{abbrv}

\end{document}